\documentclass[11pt]{article}
\usepackage{amsmath}
\usepackage{amssymb}
\usepackage{amsthm}
\usepackage{comment}
\usepackage[morefloats=200]{morefloats}
\usepackage{enumitem}
\setlist{nosep}

\newcommand{\Z}{{\mathbb Z}}
\newtheorem{theorem}{Theorem}[section]

\newtheorem{lemma}[theorem]{Lemma}

\newtheorem{construction}[theorem]{Construction}

\begin{document}
\baselineskip 18pt
\title{Group divisible designs with block size $4$ and group sizes $2$ and $5$}

{\small
\author
{R. Julian R. Abel \\
 School of Mathematics and Statistics\\
  UNSW Sydney\\ NSW 2052, Australia\\
 \texttt{r.j.abel@unsw.edu.au }
 \and
   Thomas Britz  \\
  School of Mathematics and Statistics\\
   UNSW Sydney\\ NSW 2052, Australia\\
  \texttt{britz@unsw.edu.au  }
 \and
   Yudhistira A.  Bunjamin  \\
  School of Mathematics and Statistics\\
   UNSW Sydney\\ NSW 2052, Australia\\
  \texttt{yudhi@unsw.edu.au  }
  \and
  Diana Combe  \\
  School of Mathematics and Statistics\\
   UNSW Sydney\\ NSW 2052, Australia\\
  \texttt{diana@unsw.edu.au }
}    % end author
}    % end small

\date{}

\maketitle 

\noindent{\bf Abstract:} 
In this paper we provide a $4$-GDD of type $2^2 5^5$, thereby solving the existence question for the last remaining feasible type for a $4$-GDD with no more than $30$ points.
We then show that $4$-GDDs of type $2^t 5^s$ exist for all but a finite specified set of feasible pairs $(t,s)$.

%In this paper, we show that $4$-GDDs of type $2^t 5^s$ exist for all but a small and finite set of feasible values of $s$ and $t$. We also provide a $4$-GDD of type $2^2 5^5$ which is the last previously unknown $4$-GDD with at most $30$ points.



\section{Introduction}

% Let $X$ be a finite set of elements called {\it points}, and let 
% and $\mathcal{G}$ be a partition of $X$ into  {\it groups}. 
% Also, let $\mathcal{B}$ be a collection of $k$-element subsets of $X$  called {\it blocks}. 
% A  {\it group divisible design} with block size $k$, or $k$-GDD, is a triple 
%  $(X, \mathcal{G}, \mathcal{B})$  satisfying the following properties:
% \vspace{-0.2cm}
% \begin{itemize}
%  \item no two points from the same group appear together in any block and
%  \item any two points from distinct groups appear together in exactly one block.
% \end{itemize}

% The {\it group type} (or {\it type}) of a $k$-GDD is the multiset  \{$|G|: G \in \mathcal{G}$\}.  
% Often we use `exponential' notation for group types rather than set notation. 
% In exponential notation, the type $t_1^{u_1}  t_2^{u_2} \ldots t_m^{u_m}$ means
% there are $u_i$ groups of size $t_i$ for $i=1,2, \dots, m$. 

% %A $k$-GDD of type $g^k$ is commonly called  a {\it transversal design}, denoted TD$(k,g)$. 

% Several authors have looked at existence of various infinite families of $4$-GDDs. For example, $4$-GDDs of types $g^p$ and $g^p n^1$ have been extensively studied; 
% see for instance~\cite{Forbes1, Forbes2, Forbes3,gerees, gezhu, reesk=45, gegum}. 
% Also, Abel et al. \cite{5423, ABC.50less, ABBC.2t8s} have determined the existence of $4$-GDDs of types $3^t 6^s$, $3^t 9^s$ and $2^t 8^s$ for all but a small finite number of specified cases.

A {\it group divisible design} (GDD) is a block design in which the point set, or set of treatments, is partitioned into parts. The parts are traditionally  called {\it groups}. In an experimental context, the points in the same group  might be treatments that are similar in some way, or different intensities of the same type of treatment. For convenience we denote  the point set by $X$ and the partition into groups by $\mathcal{G}$. 

Let $\mathcal{B}$ be a collection of $k$-element subsets of $X$  called {\it blocks}. Then the triple $(X, \mathcal{G}, \mathcal{B})$  is called  a   {\it group divisible design} with block size $k$, or $k$-GDD, if ($1$) no block intersects any group more than once and ($2$) any two points from distinct  groups appear together in exactly one block. If all blocks of a GDD have the same size, $k$, then we call the design a $k$-GDD.

The {\it group type} (or {\it type}) of a $k$-GDD is the multiset  \{$|G|: G \in \mathcal{G}$\}.  
Often we use `exponential' notation for group types rather than set notation. 
In exponential notation, the type $t_1^{u_1}  t_2^{u_2} \ldots t_m^{u_m}$ means
there are $u_i$ groups of size $t_i$ for $i=1,2, \dots, m$. 

In this paper we are particularly interested in group divisible designs with block size $4$. 

Various authors have looked at existence of various infinite families of $4$-GDDs. For example, $4$-GDDs of types $g^p$ and $g^p n^1$ have been extensively studied; 
see for instance~\cite{Forbes1}, \cite{Forbes2}, \cite{Forbes3}, \cite{gerees}, \cite{gezhu}, \cite{reesk=45} and \cite{gegum}. 
Also, Abel et al. \cite{5423}, \cite{ABC.50less}, \cite{ABBC.2t8s} have determined the existence of $4$-GDDs of types $3^t 6^s$, $3^t 9^s$ and $2^t 8^s$ for all but a small finite number of specified cases.

\subsection{Necessary conditions}

There are known necessary conditions for existence of a $4$-GDD of type $\{g_1, g_2, \ldots, g_m\}$.
These are given in Theorem~\ref{necessary}.  

\begin{theorem}{\rm\cite{5423, ABC.50less, krestin, reesk=45}} \label{necessary}
Suppose that there exists a $4$-GDD of type $\{g_1, g_2, \ldots, g_m\}$ and $g_1 \geq g_2 \geq \cdots \geq g_m > 0$.
Set $v= \sum_{i=1}^m g_i$. 
Then
 \begin{enumerate}
  \item $m \geq 4;$ 
  \item $v -g_i  \equiv 0\pmod{3}$ for $i=1,2, \ldots, m;$ 
  \item $\sum_{i=1}^m g_i(v-g_i)  \equiv 0\pmod{4};$ 
  \item $3g_i + g_j \leq v$ for all $i,j \in \{1,2, \ldots, m\}$, $i \neq j;$ 
  \item if $m=4$, then $g_i = g_j$ for all $i,j \in \{1,2, \ldots, m\};$
  \item if $m=5$, then the group type is of the form $h^4 n^1$ where $n \leq 3h/2;$ 
  \item if $3g_1+g_2=v$  and   $g_1  >  g_2$, then $g_3 \leq 2g_1/3;$ 
  \item if the group type is of the form  $h_1^1$ $h_2^x$ $h_3$ $h_4  \ldots h_n$ 
        where $3h_1 + h_2 = v$ and $h_2 > h_3 \geq h_4 \geq \cdots \geq h_n >0$, 
        then $n \geq 6$. 
        If further $n=6$, 
        then $h_i (h_2-h_i)  = h_j(h_2 - h_j)$ for all $i,j \in \{3,4,5,6\}$. 

 \end{enumerate}
\end{theorem}

Counting point-block pairs gives $\sum_{i=1}^m g_i(v-g_i) \equiv 0 \pmod{12}$. 
Combined  with $v-g_i \equiv 0\pmod{3}$, 
this usefully simplifies to Condition 3 in Theorem~\ref{necessary}. %, i.e. $\sum_{i=1}^m g_i(v-g_i) \equiv 0 \pmod{4}$.
Conditions 7 and 8 in Theorem~\ref{necessary} come from  \cite{ABC.50less}.

The necessary conditions in Theorem~\ref{necessary} are not sufficient conditions. 
We say that a multiset $\{g_1, g_2, \ldots, g_m\}$ of positive integers is a {\it feasible} group type for a $4$-GDD if it satisfies the conditions of Theorem~\ref{necessary}. 

In~\cite{krestin}, Kreher and Stinson tabulated the feasible types for $4$-GDDs with no more than $30$ points and answered the existence question for all but $3$ of the feasible types.  In~\cite{5423}, Abel et al. constructed designs for $2$ of the remaining unknown cases, thus completing the proof of  Theorem~\ref{gddlt30}.

\begin{theorem} {\rm\cite {5423, krestin}} \label{gddlt30}
The only feasible group types for a $4$-GDD on at most $30$ points are 
$1^4$,~$2^4$, $3^4$, $1^{13}$, $1^{9}4^1$, $2^7$, $3^5$,
$1^{16}$, $1^{12} 4^1$, $1^{8} 4^2$, $1^{4} 4^3$, $4^4$, $2^6 5^1$, $2^{10}$, $5^4$, $3^5 6^1$,
$1^{15} 7^1$, $2^9 5^1$, $3^8$, $3^4 6^2$, $6^4$, $1^{25}$, $1^{21} 4^1$, $1^{17} 4^2$, $1^{13}4^3$,
$1^9 4^4$, $1^5 4^5$, $1^1 4^6$, $2^{13}$, $2^3 5^4$, $2^9 8^1$, $3^9$, $3^5 6^2$, $3^1 6^4$,
$1^{28}$, $1^{24} 4^1$, $1^{20} 4^2$, $1^{16}4^3$, $1^{12} 4^4$,
$1^8 4^5$, $1^4 4^6$, $4^7$, $1^{14} 7^2$, $1^{10} 4^1 7^2$, $1^6 4^2 7^2$,
$1^2 4^3 7^2$, $7^4$, $2^{12} 5^1$, $2^2 5^5$, $2^8 5^1 8^1$, $3^8 6^1$, $3^4 6^3$, $6^5$, $3^7 9^1$.  
A~$4$-GDD exists for all these types except for types $2^4$, $2^6 5^1$, $6^4$ and possibly for type $2^2 5^5$.
\end{theorem}

This paper investigates the existence of $4$-GDDs of type $2^t 5^s$. In Section~\ref{section.2255}, we provide a $4$-GDD of type $2^2 5^5$, thereby completing the existence results for the table of feasible types given by Kreher and Stinson in \cite{krestin} and paving the way for our further results on $4$-GDDs with group sizes $2$ and $5$. 
In Section~\ref{knowngdds}, we state results regarding $4$-GDDs of types $g^p$ and $g^p n^1$ which we will use in this paper.
In Section~\ref{directcon.automorphisms}, direct constructions are given for a number of $4$-GDDs of type $2^t 5^s$ as well as one of type $5^8 14^1 20^1$. Finally, in Section~\ref{recurs}, we give some general recursive construction methods and use them to obtain $4$-GDDs of type $2^t 5^s$ for all but a finite number of values of $s$ and $t$.

\section{A 4-GDD of type $2^2 5^5$} \label{section.2255}

In this section we give a direct construction for a $4$-GDD of type  $2^2 5^5$. 
This construction  was found by initially using  counting arguments and then making a computer search without assuming any automorphisms. Using the program Nauty by McKay and Piperno \cite{mckay.nauty}, it was confirmed that this design has no non-trivial automorphisms.

\begin{lemma} \label{2255} 
There exists a $4$-GDD of type  $2^2 5^5$.
\end{lemma}

\begin{proof}
The point set for this design is $\{1,2, \ldots, 29\}$ and the groups are
$1-5$, $6-10$, $11-15$, $16-20$, $21-25$, $26-27$, $28-29$. The blocks are labelled as $1-59$.
For each point, we list the blocks which contain it in Table~\ref{tab.2255}.
\end{proof}

  \begin{table}[htb!]
  \begin{center}\small
  \caption{ Blocks containing each point for the $4$-GDD of type $2^2 5^5$ in Lemma~\ref{2255}. }   \label{tab.2255}
  \end{center}
%   \end{table}
\vspace{-0.7cm}

 {\small
 \noindent
 \begin{center} 
 \tabcolsep 5.6 pt
 \begin{tabular}{|l l|l l|}
\hline
  Point & Block numbers               & Point & Block numbers              \\ 
 \hline
 $1 $   & $(1,2,3,4,5,6,7,8)$         & $16$  & $(1,11,21,32,36,46,57,58)$  \\
 \hline
 $2 $   & $( 9,10,11,12,13,14,15,16)$ & $17$  & $(3, 9,23,28,38,44,55,59)$  \\
 \hline
 $3 $   & $(17,18,19,20,21,22,23,24)$ & $18$  & $(4,14,18,30,40,42,49,54)$  \\
 \hline
 $4 $   & $(25,26,27,28,29,30,31,32)$ & $19$  & $(5,12,24,26,39,41,48,56)$  \\
 \hline
 $5 $   & $(33,34,35,36,37,38,39,40)$ & $20$  & $(8,15,22,29,34,43,47,50)$  \\
 \hline
        &                             &       &                             \\
 \hline
 $6$    & $(1,9,17,25,33,41,42,43)$   & $21$  & $(2,11,22,30,37,41,52,59)$  \\
 \hline
 $7 $   & $(2,10,18,26,34,44,45,46)$  & $22$  & $(3,16,17,32,34,51,54,56)$  \\
 \hline
 $8 $   & $(3,11,19,27,35,47,48,49)$  & $23$  & $(6,13,20,26,38,43,49,58)$  \\
 \hline
 $9 $   & $(4,12,20,28,36,50,51,52)$  & $24$  & $(7,14,24,28,33,46,47,53)$  \\
 \hline
 $10$   & $(5,13,21,29,37,53,54,55)$  & $25$  & $(8,12,19,25,40,45,55,57)$  \\
 \hline
        &                             &       &                             \\
 \hline
 $11$  & $(1,10,19,30,38,50,53,56)$   & $26$  & $(4,15,17,27,39,45,53,58,59)$ \\
 \hline
 $12$  & $(2, 9,20,31,39,47,54,57)$   & $27$  & $(8,13,23,31,35,42,46,52,56)$ \\
 \hline
 $13$  & $(5,14,22,25,35,44,51,58)$   &       &                               \\
 \hline
 $14$  & $(6,15,18,32,33,48,52,55)$   & $28$  & $(6,16,24,27,37,42,44,50,57)$ \\
 \hline
 $15$  & $(7,16,23,29,36,41,45,49)$   & $29$  & $(7,10,21,31,40,43,48,51,59)$ \\
 \hline
  \end{tabular}
 \end{center}
}
\end{table}

We can now give the following update of Theorem~\ref{gddlt30}.

\begin{theorem} \label{gddlt30.complete}
The only feasible group types for a $4$-GDD on at most $30$ points are 
$1^4$,~$2^4$, $3^4$, $1^{13}$, $1^{9}4^1$, $2^7$, $3^5$,
$1^{16}$, $1^{12} 4^1$, $1^{8} 4^2$, $1^{4} 4^3$, $4^4$, $2^6 5^1$, $2^{10}$, $5^4$, $3^5 6^1$,
$1^{15} 7^1$, $2^9 5^1$, $3^8$, $3^4 6^2$, $6^4$, $1^{25}$, $1^{21} 4^1$, $1^{17} 4^2$, $1^{13}4^3$,
$1^9 4^4$, $1^5 4^5$, $1^1 4^6$, $2^{13}$, $2^3 5^4$, $2^9 8^1$, $3^9$, $3^5 6^2$, $3^1 6^4$,
$1^{28}$, $1^{24} 4^1$, $1^{20} 4^2$, $1^{16}4^3$, $1^{12} 4^4$,
$1^8 4^5$, $1^4 4^6$, $4^7$, $1^{14} 7^2$, $1^{10} 4^1 7^2$, $1^6 4^2 7^2$,
$1^2 4^3 7^2$, $7^4$, $2^{12} 5^1$, $2^2 5^5$, $2^8 5^1 8^1$, $3^8 6^1$, $3^4 6^3$, $6^5$, $3^7 9^1$.  
A~$4$-GDD exists for all these types with the definite exception of types $2^4$, $2^6 5^1$ and $6^4$.
\end{theorem}

\section{$4$-GDDs of types $g^p$ and $g^p n^1$} \label{knowngdds}

Extensive work has been done on existence of $4$-GDDs of types $g^p$ and $g^p n^1$. For type $g^p$, existence has been completely determined as given in Lemma~\ref{gp}.

\begin{lemma}{\rm\cite {BSH}}\label{gp}
There exists a $4$-GDD of type $g^p$  if and only if  
$(1)$  $p \geq 4$;   $(2)$  $g(p-1) \equiv 0\pmod{3}$;  $(3)$  $g^2p(p-1) \equiv 0\pmod{12}$ and $(4)$  $(g,p) \notin \{(2,4), (6,4)\}$.
\end{lemma}

For type $g^p n^1$ with $p \geq 4$, necessary conditions for existence are $g \equiv n\pmod{3}$, $gp \equiv 0\pmod{3}$,
$gp\big(g(p-1) + 2n\big) \equiv 0\pmod{4}$ and $0 \leq n \leq g(p-1)/2$.  
The existence of such designs has been determined in many cases;
see for instance Forbes~\cite{Forbes1}, \cite{Forbes2}, \cite{Forbes3}, Ge et al.~\cite{gezhu}, Ge and Rees~\cite{gerees}, Wei and Ge~\cite{gegum}, and Rees~\cite{reesk=45}.  
However, for a number of these feasible types the question of the existence of a corresponding design remains unanswered.

In this paper we make use of the existence results given in Lemmas~\ref{gpn1.0mod6} to \ref{5pn1}.

\begin{lemma} {\rm~\cite{Forbes1, gegum}} \label{gpn1.0mod6}
%\begin{itemize} \item     
Suppose that $g \equiv 0\pmod{6}$, $g \geq 6$, $p \geq 4$, $n \equiv 0\pmod{3}$ and $0 \leq n \leq g(p-1)/2$.
Then there exists a $4$-GDD of type $g^p n^1$ except when $(g,p,n) = (6,4,0)$.
%\item  Suppose that $g \equiv 3\pmod{6}$, $3 \leq g \leq 33$,  $p \geq 4$, $0 \leq n \leq g(p-1)/2$,
%and either (1) $p \equiv 0\pmod{4}$,   (2) $p \equiv 1\pmod{4}$,  $n \equiv 0\pmod{2}$  or (3) $p \equiv 3\pmod{4}$,  $n \equiv 1\pmod{2}$.
%Then there exists a $4$-GDD of type $g^p n^1$ except when $(g,p,n) = (6,4,0)$.
%\end{itemize}
% The  itemized results may only be needed for the 5s2t paper.
\end{lemma}

\begin{lemma}{\rm\cite{Forbes1, Forbes3, gegum}} \label{2pn1}
Suppose that $g \in \{2,20\}$, $p \geq 4$, $p \equiv 0 \pmod{3}$, $n \equiv 2\pmod{3}$ and $2 \leq n \leq g(p-1)/2$.
Then there exists a $4$-GDD of type $g^p n^1$.
\end{lemma}

\begin{lemma}{\rm\cite{Forbes2, gegum}} \label{3pn1}
Suppose that $g \equiv 3 \pmod{6}$, $g \leq 27$, $p \geq 4$ and either %, $n \equiv 0\pmod{3}$ and $0 \leq n \leq g(p-1)/2$.
\begin{itemize}
    \item $p \equiv 0\pmod{4}$, $n \equiv 0\pmod{3}$ and $0 \leq n \leq (g(p-1)-3)/2$, or
    \item $p \equiv 1\pmod{4}$, $n \equiv 0\pmod{6}$ and $0 \leq n \leq g(p-1)/2$, or
    \item $p \equiv 3\pmod{4}$, $n \equiv 3\pmod{6}$ and $3 \leq n \leq g(p-1)/2$.
\end{itemize}
Then there exists a $4$-GDD of type $g^p n^1$.
\end{lemma}

\begin{lemma}{\rm\cite{Forbes1, gegum}} \label{5pn1} % Forbes1 Thm 1.2
Suppose that $p \geq 4$ and either
\begin{itemize}
    \item $p \equiv 0\pmod{12}$, $n \equiv 2\pmod{3}$ and $2 \leq n \leq 5(p-1)/2$, or
    \item $p \equiv 3\pmod{12}$, $n \equiv 5\pmod{6}$ and $5 \leq n \leq 5(p-1)/2$, or
    \item $p \equiv 9\pmod{12}$, $n \equiv 2\pmod{6}$ and $2 \leq n \leq 5(p-1)/2$.
\end{itemize}
Then there exists a $4$-GDD of type $5^p n^1$.
\end{lemma}

\section{$4$-GDDs of type $2^t 5^s$} \label{5s2t.constructions}

From Theorem~\ref{gddlt30}, for $4$-GDDs of type $2^t 5^s$ on up to $30$ points, the only feasible group types for which it is known that there do not exist corresponding $4$-GDDs are $2^4$ and $2^6 5^1$.

In Lemma~\ref{necess2t5s}, we note two necessary conditions for the existence of a $4$-GDD of type $2^t 5^s$.

\begin{lemma} \label{necess2t5s}
If a $4$-GDD of type $2^t 5^s$ exists, then $t+s \equiv 1 \pmod{3}$ and $s \equiv 0$ or $1 \pmod{4}$.
\end{lemma}

\begin{proof}
%The necessary condition that $s+t \equiv 1 \pmod{3}$ follows from Theorem~\ref{necessary}(2) and the necessary condition $s \equiv 0$ or $1 \pmod{4}$ follows from Theorem~\ref{necessary}(3).
%From Theorem~\ref{necessary}(2)  we have  $(5s+2t) - 2 \equiv 0$   $(\bmod\ 3)$. Hence $2s+2t+1\equiv 0$  $(\bmod\ 3)$.  However, since  $2s+2t+1 = 2(s+t)+1$, it follows that  $s+t \equiv 1$ $(\bmod\ 3)$.
The necessary condition  $t+s \equiv 1 \pmod{3}$ follows from Theorem~\ref{necessary}(2) since $v-g_i \in \{2(t-1) +5s,  2t+5(s-1)\}$ and both these values are equivalent to $2(t+s-1) \pmod{3}$. The necessary condition $s \equiv 0$ or $1 \pmod{4}$ follows from Theorem~\ref{necessary}(3) since 
\begin{align*}
      \sum_{i=1}^m (g_i)(v-g_i) 
  &= 2t\big((2(t-1) + 5s\big) + 5s\big(2t + 5(s-1)\big)\\
  & = 4t(t-1) + 20st + 25s(s-1) \\
  & \equiv s(s-1) \pmod{4}\,.\qedhere
\end{align*}
\end{proof}

If $t+s > 4$ and either $s$ or $t$ lies in $\{0,1\}$, 
then the necessary conditions in Theorem~\ref{necessary} are known to be sufficient (see Lemmas~\ref{gp}, \ref{2pn1} and \ref{5pn1}). If $t+s = 4$, they are sufficient only if $(t,s) = (0,4)$.
It remains to investigate the existence of $4$-GDDs of type $2^t 5^s$ with $t \geq 2$ and $s \geq 2$.

\subsection{Direct constructions using automorphism groups} \label{directcon.automorphisms}

For the direct constructions in this section, we assume the existence of a cyclic automorphism group, $\mathbb{G}$.
We then obtain a number of suitable base blocks by computer. In each case, the point set of the design consists of one or more copies of the group~$\mathbb{G}$ 
(where  $\mathbb{G}$ is $\mathbb{Z}_4$, $\mathbb{Z}_7$, $\mathbb{Z}_8$, $\mathbb{Z}_9$, $\mathbb{Z}_{10}$ or $\mathbb{Z}_{20}$) plus possibly one copy of $\mathbb{Z}_2$ or $\mathbb{Z}_5$ and one or more infinite points.
Blocks in these designs are obtained by developing the subscripts of the points from copies of $\mathbb{G}$ over $\mathbb{G}$; the infinite points remain unaltered when developed.
When the point set includes any copies of $\mathbb{Z}_2$ or $\mathbb{Z}_5$, those points are developed over $\mathbb{Z}_2$ or $\mathbb{Z}_5$ as the others are developed over~$\mathbb{G}$.

In the past, most designs have been found using larger automorphism groups. For instance, $4$-GDDs of type $g^p n^1$ in \cite{Forbes1}, \cite{Forbes2}, \cite{Forbes3} and \cite{gegum} frequently had an automorphism group of order $gp$ or $gp/2$. Some $4$-GDDs with small automorphism groups (frequently with order $\leq 12$) can be found in \cite{5423}, \cite{ABC.50less}, \cite{ABBC.2t8s}, and in Lemma~2.1 of \cite{Forbes1}, \cite{Forbes2} and \cite{Forbes3}.

Theorem~\ref{gdd50less.directcon} summarises some known results for 4-GDDs of type $2^t 5^s$
which were obtained from direct constructions in \cite{ABC.50less} and \cite{ABBC.2t8s}. Theorem~\ref{directcon} summarises the direct constructions that are given in this paper; these are grouped according to their number of points, $v$. 

\begin{theorem} \label{gdd50less.directcon}
There exist $4$-GDDs of the types $2^6 5^4$, $2^5 5^5$, $2^9 5^4$, $2^8 5^5$, $2^{12} 5^4$, $2^2 5^8$, $2^{11} 5^5$, $2^{15} 5^4$ and $2^5 5^8$.

\begin{proof}
These designs are given in \cite{ABC.50less} except for type $2^2 5^8$ which is given in \cite{ABBC.2t8s}.
\end{proof}
\end{theorem}

\begin{theorem}\label{directcon}
There exists a $4$-GDD of type $5^8 14^1 20^1$. 
There also exist $4$-GDDs of the following types$:$
\begin{itemize}
    \item $v=53:$ $2^{14} 5^5$, $2^4 5^9;$ 
    \item $v=56:$ $2^{18} 5^4$, $2^8 5^8;$ 
    \item $v=59:$ $2^{17} 5^5;$ 
    \item $v=62:$ $2^{11} 5^8;$ 
    \item $v=65:$ $2^{20} 5^5;$ 
    \item $v=68:$ $2^{14} 5^8;$ 
    \item $v=77:$ $2^{26} 5^5$, $2^{16} 5^9$, $2^6 5^{13};$ 
    \item $v=83:$ $2^{19} 5^9.$
\end{itemize}

\begin{proof}
The designs are given in the Appendix in Tables~\ref{21455} to \ref{58141201}.
\end{proof}
\end{theorem}

\subsection{Recursive Constructions} \label{recurs}

One tool for constructing GDDs from other GDDs is that of `filling in groups'.  
This is given in Construction~\ref{fillin} in the form applicable to $k$-GDDs.  
A proof of this construction result can be found in~\cite[Theorem 1.4.12]{fmy}.

\begin{construction}\label{fillin}
Suppose that a $k$-GDD of type $\{g_1, g_2, \ldots, g_m\}$ exists and let $u$ be a non-negative integer.
Suppose also that, for each integer $i=1,2, \ldots, m-1$, there exists a $k$-GDD of type
$\{g(i,1),g(i,2),\ldots,g(i,s_i),u\}$ for which $\sum_{j=1}^{s_i} g(i,j)=g_i$.
Then there exists a $k$-GDD of type $\{g(1,1), g(1,2),\ldots, g(1,s_1), g(2,1), g(2,2),\ldots, g(2,s_2),\ldots, g(m-1,1), g(m-1,2),\ldots, g(m-1,s_{m-1}), g_m + u\}$.
If further, there exists a $k$-GDD on $g_m + u$ points of type  $\{g(m,1), g(m,2), \ldots,$ $g(m,s_m)\}$,
then there exists a $k$-GDD of type $\{g(1,1), g(1,2), \ldots, g(1,s_1), g(2,1), 
g(2,2), \ldots, g(2,s_2),$ $\ldots, g(m,1), g(m,2), \ldots, g(m,s_m)\}$.
\end{construction}

\begin{lemma} \label{fillin.from.direct}
There exist $4$-GDDs of types $2^{17} 5^8$ and $2^7 5^{12}$.

\begin{proof} 
Start with a $4$-GDD of type  $5^8 14^1 20^1$ in Table~\ref{58141201} in the Appendix and apply Construction~\ref{fillin} with $u=0$, filling in the group of size $14$ with a $4$-GDD of types $2^7$ and the group of size $20$ with $4$-GDDs of types $2^{10}$ and $5^4$ for types $2^{17} 5^8$ and $2^7 5^{12}$, respectively.
\end{proof}
\end{lemma}

\begin{lemma} \label{fillins.from.gp}
There exist $4$-GDDs of types $2^{30} 5^4$, $2^{20} 5^8$, $2^{10} 5^{12}$ and $2^{36} 5^4$.

\begin{proof}
For types $2^{30} 5^4$, $2^{20} 5^8$ and $2^{10} 5^{12}$, start with a $4$-GDD of type $20^4$ which exists by Lemma~\ref{gp} and apply Construction~\ref{fillin} with $u=0$, filling in one, two or three groups of size $20$ with a $4$-GDD of type $5^4$ and the remaining groups of size $20$ with a $4$-GDD of type $2^{10}$.

For type $2^{36} 5^4$, start with a $4$-GDD of type $18^5$ which exists by Lemma~\ref{gp} and apply Construction~\ref{fillin} with $u=2$, filling in one  group of size $18$ with a $4$-GDD of type $5^4$ and the remaining groups of size $18$ with a $4$-GDD of type $2^{10}$.
\end{proof}
\end{lemma}

\begin{lemma} \label{fillins.from.2pn1}
There exist $4$-GDDs of types $2^{21} 5^4$, $2^{24} 5^4$, $2^{27} 5^4$, $2^{33} 5^4$, $2^{32} 5^5$ and $2^{35} 5^5$.

\begin{proof}
For types $2^{21} 5^4$, $2^{24} 5^4$, $2^{27} 5^4$, $2^{33} 5^4$, $2^{32} 5^5$ and $2^{35} 5^5$, start respectively with $4$-GDDs of types $2^{21} 20^1$, 
$2^{24} 20^1$, $2^{27} 20^1$, $2^{33} 20^1$, $2^{30} 29^1$ and $2^{33} 29^1$ which exist by Lemma~\ref{2pn1}. Now, apply Construction~\ref{fillin} with $u=0$ to these GDDs, filling in the groups of sizes $20$ and $29$ with $4$-GDDs of types $5^4$ and $2^2 5^5$, respectively.
\end{proof}
\end{lemma}

\begin{lemma} \label{fillins.from.5pn1}
There exist $4$-GDDs of types $2^7 5^9$, $2^{10} 5^9$, $2^9 5^{13}$, $2^{13} 5^{12}$ and $2^3 5^{16}$.

\begin{proof}
For types $2^7 5^9$, $2^{10} 5^9$, $2^9 5^{13}$, $2^{13} 5^{12}$ and $2^3 5^{16}$, start respectively with $4$-GDDs of types $5^9 14^1$, 
$5^9 20^1$, $5^{12} 23^1$, $5^{12} 26^1$ and $5^{12} 26^1$ which exist by Lemma~\ref{2pn1}. Now, apply Construction~\ref{fillin} with $u=0$ to these GDDs, filling in the groups of sizes $14$, $20$, $23$ with $4$-GDDs of type $2^7$, $2^{10}$ and $2^9 5^1$, respectively, and the groups of size $26$ with a $4$-GDD of type $2^{13}$ for type $2^{13} 5^{12}$ and a $4$-GDD of type $2^3 5^4$ for type $2^3 5^{16}$.
\end{proof}
\end{lemma}

It follows from Lemma~\ref{necess2t5s} that for $4$-GDDs of type $2^t 5^s$ we have  $2t+5s \equiv 2 \pmod{3}$. When considering constructions and infinite families of these GDDs with $98$ or more points, we consider the twenty cases $2t+5s \equiv 2, 5, \ldots, 59 \pmod{60}$ separately.

\begin{lemma} \label{seq}
If $v=2t+5s$, $t + s \equiv 1 \pmod{3}$ and $s \equiv 0$ or $1 \pmod{4}$, then a $4$-GDD of type $2^t 5^s$ exists if $98 \leq v \leq 200$ except possibly for $(t,s) \in \{(11,17), (4,21)\}$.

\begin{proof}
If $t\in \{0,1\}$ or $s \in \{0,1\}$, then existence is given by Lemmas~\ref{gp}, \ref{2pn1} and~\ref{5pn1}.
For the remaining values of $t$ and $s$, constructions are given in Table~\ref{tab.98to197}. 
For each pair $(t,s)$, we start with a given input GDD and then  apply Construction~\ref{fillin} using the given value of~$u$. 
We also give the required fill in designs for each construction. 
%Type $2^3 5^{16}$ exists by Lemma~\ref{fillins.from.5pn1}.
All input $4$-GDDs have a type of the form $g^p$ or $g^p n^1$  
(in which case they exist by Lemma~\ref{gp}, \ref{gpn1.0mod6}, \ref{2pn1}, \ref{3pn1} or \ref{5pn1}).
All fill in $4$-GDDs have $v \leq 56$ and a type of the form $2^t 5^s$ (in which case they exist by Theorem~\ref{gddlt30.complete}, \ref{gdd50less.directcon}  or \ref{directcon}).
\end{proof}
\end{lemma}

\begin{table}[ht!]
\caption{Constructions for $4$-GDDs of type $2^t 5^s$ with $98 \leq 2s+5t \leq 197$ in Lemma~\ref{seq}.}\label{tab.98to197}
{\noindent\footnotesize
\begin{center}
  \begin{tabular}{|c|l|l|c|l|} \hline
    $v$ & Types & Input 4-GDD type & $u$ & Fill in 4-GDD types \\ \hline
    
     98 & $2^{39} 5^4, 2^{29} 5^8, 2^{19} 5^{12}, 2^{9} 5^{16}$ & $24^4$ & 2 & $2^{13}, 2^3 5^4$ \\ \hline
    
    101 & $2^{38} 5^5, 2^{28} 5^9, 2^{18} 5^{13}, 2^{8} 5^{17}$ & $24^4$ & 5 & $2^{12} 5^1, 2^2 5^5$ \\ \hline
    
    104 & $2^{42} 5^4, 2^{32} 5^8, 2^{22} 5^{12}, 2^{12} 5^{16}$ & $26^4$ & 0 & $2^{13}, 2^3 5^4$ \\ \hline
        & $2^2 5^{20}$ & $5^{15} 29^1$ & 0 & $2^2 5^5$ \\ \hline
    
    107 & $2^{41} 5^5$ & $2^{36} 35^1$ & 0 & $2^5 5^5$ \\ \hline
        & $2^{31} 5^9, 2^{21} 5^{13}$ & $21^4 18^1$ & 5 & $2^{13}, 2^3 5^4, 2^9 5^1$ \\ \hline
        & $ 2^{11}5^{17}$ & Unknown &  &  \\ \hline
    
    110 & $2^{45} 5^4, 2^{35} 5^8, 2^{25} 5^{12}$ & $24^4 12^1$ & 2 & $2^{13}, 2^3 5^4, 2^7$ \\ \hline
        & $2^{15} 5^{16}, 2^5 5^{20}$ & $15^5 30^1$ & 5 & $5^4, 2^{15} 5^1, 2^5 5^5$ \\ \hline
    
    113 & $2^{44} 5^5, 2^{34} 5^9$ & $21^4 27^1$ & 2 & $2^9 5^1, 2^{12} 5^1, 2^2 5^5$ \\ \hline
        & $2^{24} 5^{13}, 2^{14} 5^{17}$ & $21^4 24^1$ & 5 & $2^3 5^4, 2^{12} 5^1, 2^2 5^5$ \\ \hline
        & $2^{4} 5^{21}$ & Unknown &  &  \\ \hline
    
    116 & $2^{48} 5^4, 2^{38} 5^8, \ldots, 2^8 5^{20}$ & $24^4 15^1$ & 5 & $2^{12} 5^1, 2^2 5^5, 5^4$ \\ \hline
    
    119 & $2^{47} 5^{5}, 2^{37} 5^9, \ldots, 2^{17} 5^{17}$ & $24^4 18^1$ & 5 & $2^{12} 5^1, 2^2 5^5, 2^9 5^1$ \\ \hline
        & $5^{21} 2^{7}$ & $5^{21} 14^1$ & 0 & $2^7$ \\ \hline
    
    122 & $2^{51} 5^4, 2^{41} 5^8, \ldots, 2^{11} 5^{20}$ & $24^4 21^1$ & 5 & $2^{12} 5^1, 2^2 5^5, 2^3 5^4$ \\ \hline
    
    125 & $2^{50} 5^{5}, 2^{40} 5^9, \ldots, 2^{10} 5^{21}$ & $24^5$ & 5 & $2^{12} 5^1, 2^2 5^5$ \\ \hline
    
    128 & $2^{54} 5^4, 2^{44} 5^8, \ldots, 2^{14} 5^{20}$ & $24^4 27^1$ & 5 & $2^{12} 5^1, 2^2 5^5, 2^6 5^4$ \\ \hline
        & $2^{4} 5^{24}$ & $27^4 18^1$ & 2 & $2^2 5^5, 5^4$ \\ \hline
    
    131 & $2^{53} 5^{5}, 2^{43} 5^9, \ldots, 2^{13} 5^{21}$ & $24^4 30^1$ & 5 & $2^{12} 5^1, 2^2 5^5, 2^5 5^5$ \\ \hline
        & $2^{3} 5^{25}$ & $15^7 21^1$ & 5 & $5^4, 2^3 5^4$ \\ \hline
    
    134 & $2^{57} 5^4, 2^{47} 5^8, \ldots, 2^{7} 5^{24}$ & $20^6 14^1$ & 0 & $2^{10}$, $5^4$, $2^7$ \\ \hline
    
    137 & $2^{56} 5^{5}, 2^{46} 5^9, \ldots, 2^{16} 5^{21}$ & $24^4 36^1$ & 5 & $2^{12} 5^1, 2^2 5^5, 2^8 5^5$ \\ \hline
        & $2^{6} 5^{25}$ & $15^7 27^1$ & 5 & $5^4, 2^6 5^4$ \\ \hline
    
    140 & $2^{60} 5^4, 2^{50} 5^8, \ldots, 2^{10} 5^{24}$ & $20^7$ & 0 & $2^{10}$, $5^4$ \\ \hline
    
    143 & $2^{59} 5^5, 2^{49} 5^9, \ldots, 2^9 5^{25}$ & $20^6 23^1$ & 0 & $2^{10}$, $5^4$, $2^9 5^1$ \\ \hline
    
    146 & $2^{63} 5^4, 2^{53} 5^8, \ldots, 2^3 5^{28}$ & $20^6 26^1$ & 0 & $2^{10}$, $5^4$, $2^3 5^4$ \\ \hline
    
    149 & $2^{62} 5^5, 2^{52} 5^9, \ldots, 2^2 5^{29}$ & $20^6 29^1$ & 0 & $2^{10}$, $5^4$, $2^2 5^5$ \\ \hline
        
    152 & $2^{66} 5^4, 2^{56} 5^8, \ldots, 2^6 5^{28}$ & $20^6 32^1$ & 0 & $2^{10}$, $5^4$, $2^6 5^4$ \\ \hline
    
    155 & $2^{65} 5^5, 2^{55} 5^9, \ldots, 2^5 5^{29}$ & $20^6 35^1$ & 0 & $2^{10}$, $5^4$, $2^5 5^5$ \\ \hline
        
    158 & $2^{69} 5^4, 2^{59} 5^8, \ldots, 2^9 5^{28}$ & $20^6 38^1$ & 0 & $2^{10}$, $5^4$, $2^9 5^4$ \\ \hline
    
    161 & $2^{68} 5^5, 2^{58} 5^9, \ldots, 2^8 5^{29}$ & $20^6 41^1$ & 0 & $2^{10}$, $5^4$, $2^8 5^5$ \\ \hline
        
    164 & $2^{72} 5^4, 2^{62} 5^8, \ldots, 2^2 5^{32}$ & $20^6 44^1$ & 0 & $2^{10}$, $5^4$, $2^{12} 5^4$, $2^2 5^8$ \\ \hline
    
    167 & $2^{71} 5^5, 2^{61} 5^9, \ldots, 2^{11} 5^{29}$ & $20^6 47^1$ & 0 & $2^{10}$, $5^4$, $2^{11} 5^5$, $2^1 5^9$ \\ \hline
    
    170 & $2^{75} 5^4, 2^{65} 5^8, \ldots, 2^5 5^{32}$ & $20^6 50^1$ & 0 & $2^{10}$, $5^4$, $2^{15} 5^4$, $2^5 5^8$ \\ \hline
        
    173 & $2^{74} 5^5, 2^{64} 5^9, \ldots, 2^{14} 5^{29}$ & $24^5 48^1$ & 5 & $2^{12} 5^1, 2^2 5^5, 2^{14} 5^5, 2^4 5^9$ \\ \hline
        & $2^4 5^{33}$ & $15^8 48^1$ & 5 & $5^4, 2^4 5^9$ \\ \hline
    
    176 & $2^{78} 5^4, 2^{68} 5^8, \ldots, 2^{18} 5^{28}$ & $24^6 27^1$ & 5 & $2^{12} 5^1, 2^2 5^5, 2^6 5^4$ \\ \hline
        & $2^8 5^{32}$ & $15^8 51^1$ & 5 & $5^4, 2^8 5^8$ \\ \hline
        
    179 & $2^{77} 5^5, 2^{67} 5^9, \ldots, 2^{17} 5^{29}$ & $24^6 30^1$ & 5 & $2^{12} 5^1, 2^2 5^5, 2^5 5^5$ \\ \hline
        & $2^7 5^{33}$ & $5^{33} 14^1$ & 0 & $2^7$ \\ \hline
    
    182 & $2^{81} 5^4, 2^{71} 5^8, \ldots, 2^{11} 5^{32}$ & $20^9 2^1$ & 0 & $2^{10}$, $5^4$ \\ \hline
        
    185 & $2^{80} 5^5, 2^{70} 5^9, \ldots, 2^{10} 5^{33}$ & $20^9 5^1$ & 0 & $2^{10}$, $5^4$ \\ \hline
        
    188 & $2^{84} 5^4, 2^{74} 5^8, \ldots, 2^{14} 5^{32}$ & $24^6 39^1$ & 5 & $2^{12} 5^1, 2^2 5^5, 2^{12} 5^4, 2^2 5^8$ \\ \hline
        & $2^4 5^{36}$ & $15^9 48^1$ & 5 & $5^4, 2^4 5^9$ \\ \hline
        
    191 & $2^{83} 5^5, 2^{73} 5^9, \ldots, 2^{13} 5^{33}$ & $24^6 42^1$ & 5 & $2^{12} 5^1, 2^2 5^5, 2^{11} 5^5, 2^1 5^9$ \\ \hline
        & $2^3 5^{37}$ & $15^{11} 21^1$ & 5 & $5^4, 2^3 5^4$ \\ \hline
    
    194 & $2^{87} 5^4, 2^{77} 5^8, \ldots, 2^7 5^{36}$ & $20^9 14^1$ & 0 & $2^{10}$, $5^4$, $2^7$ \\ \hline
        
    197 & $2^{86} 5^5, 2^{76} 5^9, \ldots, 2^{16} 5^{33}$ & $24^6 48^1$ & 5 & $2^{12} 5^1, 2^2 5^5, 2^{14} 5^5, 2^4 5^9$ \\ \hline
        & $2^6 5^{37}$ & $15^{11} 27^1$ & 5 & $5^4, 2^6 5^4$ \\ \hline
        
    200 & $2^{90} 5^4, 2^{80} 5^8, \ldots, 2^{10} 5^{36}$ & $20^{10}$ & 0 & $2^{10}, 5^4$ \\ \hline
  \end{tabular}
\end{center}}
\end{table}

\begin{lemma} \label{seq811}
If $v=2t+5s$, $t + s \equiv 1 \pmod{3}$ and $s \equiv 0$ or $1 \pmod{4}$ where $v \equiv 8 \pmod{60}$ or $v \equiv 11 \pmod{60}$, then a $4$-GDD of type $2^t 5^s$ exists if $248 \leq v \leq 371$.

\begin{proof}
If $t\in \{0,1\}$ or $s \in \{0,1\}$, then existence is given by Lemmas~\ref{gp}, \ref{2pn1} and~\ref{5pn1}.
For the remaining values of $t$ and $s$, constructions are given in Table~\ref{tab.248to371}. 
For each pair $(t,s)$, we start with a given input GDD and then  apply Construction~\ref{fillin} using the given value of~$u$. 
We also give the required fill in designs for each construction. 
All input $4$-GDDs have a type of the form $g^p n^1$ (in which case they exist by Lemma~\ref{gpn1.0mod6} or \ref{3pn1}).
All fill in $4$-GDDs have a type of the form $2^t 5^s$ and $v \leq 53$ (in which case they exist by Theorem~\ref{gddlt30.complete}, \ref{gdd50less.directcon}  or \ref{directcon}) except for types $2^{13} 5^{12}$ and $2^3 5^{16}$ (these exist by Lemma~\ref{fillins.from.5pn1}).
\end{proof}
\end{lemma}

\begin{table}[hbt!]\small
\caption{Constructions for $4$-GDDs of type $2^t 5^s$ with $248 \leq 2s+5t \leq 371$ in Lemma~\ref{seq811}.}\label{tab.248to371}
{\noindent
\begin{center}
  \begin{tabular}{|c|l|l|c|l|} \hline
    $v$ & Types & Input 4-GDD type & $u$ & Fill in 4-GDD types \\ \hline
        
    248 & $2^{114} 5^4, 2^{104} 5^8, \ldots, 2^{24} 5^{40}$ & $24^9 27^1$ & 5 & $2^{12} 5^1, 2^2 5^5, 2^6 5^4$  \\ \hline
        & $2^{14} 5^{44}, 2^{4} 5^{48}$ & $15^{13} 48^1$ & 5 & $5^4, 2^{14} 5^5, 2^4 5^9$ \\ \hline
        
    251 & $2^{113} 5^5, 2^{103} 5^9, \ldots, 2^{23} 5^{41}$ & $24^9 30^1$ & 5 & $2^{12} 5^1, 2^2 5^5, 2^5 5^5$ \\ \hline
        & $2^{13} 5^{45}, 2^3 5^{49}$ & $15^{15} 21^1$ & 5 & $5^4, 2^{13}, 2^3 5^4$ \\ \hline
        
    308 & $2^{144} 5^4, 2^{134} 5^8, \ldots, 2^{24} 5^{52}$ & $24^{11} 39^1$ & 5 & $2^{12} 5^1, 2^2 5^5, 2^{12} 5^4, 2^2 5^8$ \\ \hline
        & $2^{14} 5^{56}, 2^4 5^{60}$ & $15^{17} 48^1$ & 5 & $5^4, 2^{14} 5^5, 2^4 5^9$ \\ \hline
        
    311 & $2^{143} 5^5, 2^{133} 5^9, \ldots, 2^{23} 5^{53}$ & $24^{11} 42^1$ & 5 & $2^{12} 5^1, 2^2 5^5, 2^{11} 5^5, 2^1 5^9$ \\ \hline
        & $2^{13} 5^{57}, 2^{3} 5^{61}$ & $15^{15} 81^1$ & 5 & $5^4, 2^{13} 5^{12}, 2^3 5^{16}$ \\ \hline
    
    368 & $2^{174} 5^4, 2^{164} 5^8, \ldots, 2^{34} 5^{60}$ & $24^{14} 27^1$ & 5 & $2^{12} 5^1, 2^2 5^5, 2^6 5^4$ \\ \hline
        & $2^{24} 5^{64}, 2^{14} 5^{68}, 2^{4} 5^{72}$ & $15^{21} 48^1$ & 5 & $5^4, 2^{24} 5^1, 2^{14} 5^5, 2^4 5^9$ \\ \hline
        
    371 & $2^{173} 5^5, 2^{163} 5^9, \ldots, 2^{33} 5^{61}$ & $24^{14} 30^1$ & 5 & $2^{12} 5^1, 2^2 5^5, 2^5 5^5$ \\ \hline
        & $2^{23} 5^{65}$ & $27^{12} 45^1$ & 2 & $2^2 5^5, 2^{11} 5^5$ \\ \hline
        & $2^{13} 5^{69}, 2^{3} 5^{73}$ & $15^{19} 81^1$ & 5 & $5^4, 2^{13} 5^{12}, 2^3 5^{16}$ \\ \hline
  \end{tabular}
\end{center}}
\end{table}

We now give a construction for infinite families consisting of the remaining $4$-GDDs of type $2^t 5^s$ with at least $203$ points in Lemmas~\ref{inf.family.main} and \ref{inf.family.811}.

\begin{lemma} \label{inf.family.main}
Suppose that $t+s \equiv 1 \pmod{3}$, $s \equiv 0$ or $1 \pmod{4}$, $v=2t+5s \geq 203$ and $v \not\equiv 8$ or $11 \pmod{60}$.
Then there exists a $4$-GDD of type $2^t 5^s$.
\end{lemma}

\begin{proof}
If $s \in \{0,1\}$ or $t = 0$, 
then existence is given by Lemmas~\ref{gp} and \ref{5pn1}, 
so we may assume that $s \geq 2$ and $t > 0$.

Suppose that $\ell \in \{2, 5, \ldots, 59\}$, $\ell \neq 8$ or $11$ and $v \equiv \ell \pmod{60}$. 
Then set $m = (v-\ell)/20$ if $\ell > 20$ and $m = (v-\ell-60)/20$ otherwise.
Note that $m \equiv 0 \pmod{3}$.
Since $v \geq 203$, it follows that $m \geq 9$. 

Form a 4-GDD of type $20^m \ell^1$ if $\ell \geq 20$ and type $20^m (\ell + 60)^1$ otherwise. 
These types exist by Lemma~\ref{2pn1}. Now, we apply Construction~\ref{fillin} with $u=0$.  Note that  there are two steps to these constructions.

{\it Step 1}:

% 1
If $\ell = 23$, 
then fill in the group of size $\ell$ with a 4-GDD of type $2^9 5^1$
and set $x = (s-1)/4$.

%4
If $\ell \in \{26,32,38\}$, 
then fill in the group of size $\ell$ with a 4-GDD of type $2^3 5^4$, $2^6 5^4$ or $2^9 5^4$ for $\ell=26,32,38$, respectively,
and set $x = (s-4)/4$.

%5
If $\ell \in \{29,35,41\}$, 
then fill in the group of size $\ell$ with a 4-GDD of type $2^2 5^5$, $2^5 5^5$ or $2^8 5^5$ for $\ell=29,35,41$, respectively,
and set $x = (s-5)/4$.

If $\ell \in \{44,50,56\}$ and $s = 4$, 
then fill in the group of size $\ell$ with a 4-GDD of type $2^{12} 5^4$, $2^{15} 5^4$ or $2^{18} 5^4$ for $\ell=44,50,56$, respectively, 
and set $x = 0$. 
If $\ell \in \{44,50,56\}$ and $s \geq 8$, 
then fill in the group of size $\ell$ with a 4-GDD of type $2^2 5^8$, $2^5 5^8$ or $2^8 5^8$ for $\ell=44,50,56$, respectively, 
and set $x = (s-8)/4$.

If $\ell \in \{47,53,59\}$ and $s = 5$, 
then fill in the group of size $\ell$ with a 4-GDD of type $2^{11} 5^5$, $2^{14} 5^5$ or $2^{17} 5^5$ for $\ell=47,53,59$, respectively, and set $x = 0$.
If $\ell \in \{47,53,59\}$ and $s \geq 9$, then fill in the group of size $\ell$ with a 4-GDD of type $2^1 5^9$, $2^4 5^9$ or $2^7 5^9$ for $\ell=47,53,59$, respectively, and set $x = (s-9)/4$.

If $\ell = 5$ and $s = 5$, then fill in the group of size $\ell + 60 = 65$ with a 4-GDD of type $2^{20} 5^5$ and set $x = 0$.
If $\ell = 5$ and $s \geq 9$, then fill in the group of size $\ell + 60 = 65$ with a 4-GDD of type $2^{10} 5^9$ and set $x = (s-9)/4$.

If $\ell \in \{2,14,20\}$ and $s = 4$, then fill in the group of size $\ell + 60$ with a 4-GDD of type $2^{21} 5^4$, $2^{27} 5^4$ or $2^{30} 5^4$ for $\ell=2,14,20$, respectively, and set $x = 0$. 
If $\ell \in \{2,14,20\}$ and $s = 8$, then fill in the group of size $\ell + 60$ with a 4-GDD of type $2^{11} 5^8$, $2^{17} 5^8$ or $2^{20} 5^8$ for $\ell=2,14,20$, respectively, and set $x = 0$. 
If $\ell \in \{2,14,20\}$ and $s \geq 12$, then fill in the group of size $\ell + 60$ with a 4-GDD of type $2^1 5^{12}$, $2^7 5^{12}$ or $2^{10} 5^{12}$ for $\ell=2,14,20$, respectively, and set $x = (s-12)/4$.

If $\ell = 17$ and $s = 5$, then fill in the group of size $\ell + 60$ with a 4-GDD of type $2^{26} 5^5$ and set $x = 0$. 
If $\ell = 17$ and $s = 9$, then fill in the group of size $\ell + 60$ with a 4-GDD of type $2^{16} 5^9$ and set $x = 0$. 
If $\ell = 17$ and $s \geq 13$, then fill in the group of size $\ell + 60$ with a 4-GDD of type $2^6 5^{13}$ and set $x = (s-13)/4$.

{\it Step 2}: 

Finally, fill in $x$ groups of size $20$ with a 4-GDD of type $5^4$ and the remaining $m-x$ groups of size $20$, if any, with a 4-GDD of type $2^{10}$.
\end{proof}

\begin{lemma} \label{inf.family.811}
Suppose that $t+s \equiv 1 \pmod{3}$, $s \equiv 0$ or $1 \pmod{4}$, $v=2t+5s \geq 428$ and $v \equiv 8$ or $11 \pmod{60}$.
Then there exists a $4$-GDD of type $2^t 5^s$.
\end{lemma}

\begin{proof}
If $s \in \{0,1\}$ or $t = 0$, then existence is given by Lemmas~\ref{gp} and \ref{5pn1}, so we may assume that $s \geq 2$ and $t > 0$.

Suppose that $\ell \in \{8, 11\}$ where $v \equiv \ell \pmod{60}$. 
Then set $m = (v-\ell)/20 - 6$.
Note that $m \equiv 0 \pmod{3}$.
Since $v \geq 428$, it follows that $m \geq 15$. 
If $\ell = 8$, then form a 4-GDD of type $20^m 128^1$ and if $\ell = 11$, then form a 4-GDD of type $20^m 131^1$. 
These types exist by Lemma~\ref{2pn1}. 
Now, apply Construction~\ref{fillin} with $u=0$.

If $\ell = 8$, then for $s = 4, 8, \ldots, 24$, respectively, fill in the group of size $128$ with a $4$-GDD of type $2^{54} 5^4$, $2^{44} 5^8$, $\ldots$ , $2^4 5^{24}$, respectively. 
Also, fill in $m$ groups of size $20$ with a $4$-GDD of type $2^{10}$.
If $\ell = 8$ and $s\geq 28$, then fill in $(s-24)/4$ groups of size $20$ with a $4$-GDD of type $5^4$, the remaining $m-(s-24)/4$ groups of size $20$ with a $4$-GDD of type $2^{10}$ and the group of size $128$ with a $4$-GDD of type $2^4 5^{24}$.

If $\ell = 11$, then for $s = 5, 9, \ldots, 25$, respectively, fill in the group of size $131$ with a $4$-GDD of type $2^{53} 5^5$, $2^{43} 5^9$, $\ldots$ , $2^3 5^{25}$, respectively. 
Also, fill in $m$ groups of size $20$ with a $4$-GDD of type $2^{10}$.
If $\ell = 11$ and $s\geq 29$, then fill in $(s-25)/4$ groups of size $20$ with a $4$-GDD of type $5^4$. 
Also, fill in the remaining $m-(s-25)/4$ groups of size $20$ with a $4$-GDD of type $2^{10}$ and the group of size $131$ with a $4$-GDD of type $2^3 5^{25}$.
\end{proof}

\section{Summary}

Combining all the lemmas in Section \ref{5s2t.constructions}, we now have the following theorem.

\begin{theorem} \label{summary2t8s}
If $t,s \geq 0$ and $t+s \geq 7$, then a $4$-GDD of type $2^t 5^s$ exists if and only if $t+s \equiv 1 \pmod{3}$ 
and  $s \equiv 0$ or $1 \pmod{4}$, except when $(t,s) = (6,1)$ and possibly when
\begin{itemize}
\item  $v = 68:$ $(t,s) \in \{(4,12)\};$
\item  $v = 71:$ $(t,s) \in \{(23,5), (13,9), (3,13)\};$
\item  $v = 83:$ $(t,s) \in \{(29,5)\};$
\item  $v = 86:$ $(t,s) \in \{(23,8)\};$
\item  $v = 89:$ $(t,s) \in \{(22,9), (12,13), (2,17)\};$
\item  $v = 92:$ $(t,s) \in \{(26,8), (16,12), (6,16)\};$
\item  $v = 95:$ $(t,s) \in \{(25,9), (15,13), (5,17)\};$
\item  $v = 107:$ $(t,s) \in \{(11,17)\};$
\item  $v = 113:$ $(t,s) \in \{(4,21)\}.$
\end{itemize}
\end{theorem}

%\section{Summary}\label{summary}
% For $v \leq 30$, $4$-GDDs on $v$ points are known for all feasible group types (except for types $2^4$, $6^4$ and $2^6 5^1$ which are known not
% to exist and possibly for type $2^2 5^5$) \cite{5423, krestin}. Solutions are also known for all feasible group types  when  both $30 < v \leq 50$ and 
% $v \equiv 0  \pmod{3}$ \cite{ABC.50less}.  In the current  paper, $4$-GDDs of type $2^t 8^s$ have been found for all but a finite number 
% of  feasible  values of $t$ and $s$, and several $4$-GDDs with  $v \equiv 2  \pmod{3}$   and $30 < v \leq 50$ have been found.    However,  a number of these 
% still remain unknown as indicated in the previous section, and one is known not to exist (type $2^4 8^3$).  In addition, for $30  < v \leq 50$ and  
% $v \equiv 1  \pmod{3}$,  not many $4$-GDDs are known for types not of the form $g^p$, $g^p n^1$ or $4^s 1^t n^1$. As noted before Lemma~\ref{275482},
% there are now no unknown $4$-GDDs of type $4^s 1^t n^1$ with  $v \leq 50$. For more information on $4$-GDDs with $v \leq 50$, $v \equiv 1 \pmod{3}$,
% including a list of all feasible types, see \cite{ABC.50less}.   As indicated there, there are significantly more feasible types 
% for  $v \equiv 1  \pmod{3}$ than for $v \equiv 0$ or $2 \pmod{3}$.

\section{Acknowledgements} 
This research used the computational cluster Katana supported by Research Technology Services at UNSW Sydney.
The authors are grateful to Michael Denes for his computing support.
The third author acknowledges the support from an Australian Government Research Training Program Scholarship and from the School of  Mathematics and Statistics, UNSW Sydney.
%The authors would like to thank the reviewers for detailed checking of the constructions and providing a number of useful comments.

\section*{ORCID}
R. J. R. Abel:     https://orcid.org/0000-0002-3632-9612\\
T. Britz:          https://orcid.org/0000-0003-4891-3055\\
Y. A. Bunjamin:    https://orcid.org/0000-0001-6849-2986\\
D. Combe:          https://orcid.org/0000-0002-1055-3894

\clearpage
\section*{Appendix}

\begin{table}[ht]
\small
\caption{$4$-GDD of type $2^{14}5^{5}$} \label{21455}
\medskip
Points: $a_i,b_i,c_i,d_i,e_i,f_i,p_i,q_i,r_i,s_i,t_i,u_i,v_i$ for $i\in\Z_4$; \; $\infty$.\\
Groups: $\{a_i,b_i,c_i,d_i,e_i\}$ for $i\in\Z_4$; \, $\{f_i: i \in\Z_4\} \cup \{\infty\}$;\\\hspace*{12.5mm}
        $\{p_i,q_i\}$ for $i\in\Z_4$; \; $\{w_i, w_{i+2}\}$ for $w \in \{r,s,t,u,v\}$ and $i \in \{0,1\}$.\\
Develop the following blocks  $(\bmod\ 4)$. The first five blocks in the first column generate one block each and the sixth generates two blocks.
{\footnotesize
\[
  \begin{array}{|l|l|l|l|l|l|}\hline
    \{a_0,a_1,a_2,a_3\} & \{a_0,b_1,c_2,q_0\}    & \{a_0,f_2,q_1,v_3\} & \{b_0,f_1,p_2,u_0\}    & \{c_0,p_1,u_1,u_2\} & \{e_0,f_3,q_1,u_3\}\\\hline
    \{b_0,b_1,b_2,b_3\} & \{a_0,b_2,f_0,s_0\}    & \{a_0,f_1,r_3,s_2\} & \{b_0,f_0,t_0,v_0\}    & \{c_0,q_0,r_0,u_0\} & \{e_0,p_3,u_2,v_1\}\\\hline
    \{c_0,c_1,c_2,c_3\} & \{a_0,b_3,p_0,r_0\}    & \{a_0,q_3,r_1,u_0\} & \{b_0,f_3,t_2,v_1\}    & \{c_0,q_3,s_1,t_2\} & \{e_0,r_2,s_2,v_3\}\\\hline
    \{d_0,d_1,d_2,d_3\} & \{a_0,c_1,d_2,\infty\} & \{a_0,s_1,t_3,u_2\} & \{b_0,p_0,r_2,\infty\} & \{c_0,q_1,t_3,v_1\} & \{e_0,t_3,u_0,\infty\}\\\hline
    \{e_0,e_1,e_2,e_3\} & \{a_0,c_3,t_0,v_1\}    & \{b_0,c_3,e_1,t_3\} & \{b_0,r_0,u_2,v_2\}    & \{d_0,e_2,f_0,q_0\} & \{f_0,q_1,r_0,t_2\}\\\hline
    \{p_0,p_2,q_1,q_3\} & \{a_0,d_1,p_1,t_2\}    & \{b_0,c_2,r_3,s_1\} & \{c_0,d_2,f_3,p_3\}    & \{d_0,e_3,f_3,t_0\} & \{p_0,q_2,s_2,v_1\}\\\hline
                        & \{a_0,d_3,q_2,u_1\}    & \{b_0,d_3,q_0,s_3\} & \{c_0,d_3,s_2,v_0\}    & \{d_0,e_1,r_1,r_2\} & \{p_0,r_3,t_0,t_3\}\\\hline
                        & \{a_0,e_1,p_3,t_1\}    & \{b_0,d_2,t_1,u_3\} & \{c_0,e_1,f_2,s_0\}    & \{d_0,f_2,s_1,u_3\} & \{q_0,s_1,v_1,\infty\}\\\hline
                        & \{a_0,e_3,r_2,s_3\}    & \{b_0,d_1,u_1,v_3\} & \{c_0,e_3,p_0,v_3\}    & \{d_0,p_2,p_3,s_2\} & \{s_0,s_1,t_0,u_0\}\\\hline
                        & \{a_0,e_2,u_3,v_0\}    & \{b_0,e_3,p_3,s_0\} & \{c_0,f_0,p_2,r_3\}    & \{d_0,q_2,r_3,t_2\} & \\\hline
                        & \{a_0,f_3,p_2,v_2\}    & \{b_0,e_2,q_1,q_2\} & \{c_0,f_1,r_2,u_3\}    & \{d_0,r_0,v_0,v_3\} & \\\hline
  \end{array}
\]}
\end{table}

\begin{table}[ht!]
\small
\caption{$4$-GDD of type $2^{4}5^{9}$} \label{2459}
\medskip
Points: $a_i,b_i,c_i,d_i,e_i,f_i,p_i,q_i,r_i,s_i,t_i,u_i,v_i$ for $i\in\Z_4$; \; $\infty$.\\
Groups: $\{a_i, b_i, c_i, d_i, e_i\}$ for $i\in\Z_4$; \, $\{ f_i:  i \in\Z_4\} \cup \{\infty \}$; \, $\{r_i, s_i, t_i, u_i, v_i\}$ for $i\in\Z_4$;\\\hspace*{12.5mm} 
  $\{w_i, w_{i+2}\}$ for $w \in \{p,q\}$ and $i \in \{0,1\}$.\\
Develop the following blocks  $(\bmod\ 4)$. The ten blocks in the first column generate one block each.
{\footnotesize
\[
  \begin{array}{|l|l|l|l|l|l|}\hline
    \{a_0,a_1,a_2,a_3\} & \{a_0,b_1,f_0,q_0\}    & \{a_0,f_2,p_3,v_2\} & \{b_0,f_2,p_1,s_0\}    & \{c_0,p_1,r_2,t_1\} & \{e_0,f_3,s_2,u_0\}\\\hline
    \{b_0,b_1,b_2,b_3\} & \{a_0,b_2,p_0,v_0\}    & \{a_0,f_1,q_3,u_3\} & \{b_0,f_1,t_2,v_3\}    & \{c_0,q_2,r_3,v_0\} & \{e_0,p_0,s_0,t_1\}\\\hline
    \{c_0,c_1,c_2,c_3\} & \{a_0,b_3,r_0,s_1\}    & \{a_0,f_3,r_3,t_1\} & \{b_0,f_0,u_3,v_1\}    & \{c_0,q_1,s_2,t_0\} & \{e_0,q_2,r_0,v_2\}\\\hline
    \{d_0,d_1,d_2,d_3\} & \{a_0,c_1,d_2,\infty\} & \{a_0,q_2,r_2,v_1\} & \{b_0,p_0,p_3,u_0\}    & \{d_0,e_3,f_1,p_1\} & \{e_0,t_2,v_1,\infty\}\\\hline
    \{e_0,e_1,e_2,e_3\} & \{a_0,c_2,p_1,t_0\}    & \{b_0,c_2,d_1,v_0\} & \{b_0,q_0,s_3,\infty\} & \{d_0,e_2,f_3,r_0\} & \{f_0,p_2,q_3,t_0\}\\\hline
    \{r_0,r_1,r_2,r_3\} & \{a_0,c_3,s_0,u_1\}    & \{b_0,c_3,e_1,q_2\} & \{c_0,d_2,f_2,s_3\}    & \{d_0,e_1,p_2,v_0\} & \{p_0,q_3,r_2,s_1\}\\\hline
    \{s_0,s_1,s_2,s_3\} & \{a_0,d_3,p_2,r_1\}    & \{b_0,c_1,t_0,u_2\} & \{c_0,e_1,f_1,r_0\}    & \{d_0,f_2,q_3,t_1\} & \{p_0,r_0,u_3,\infty\}\\\hline
    \{t_0,t_1,t_2,t_3\} & \{a_0,d_1,t_3,u_2\}    & \{b_0,d_2,r_3,s_1\} & \{c_0,e_3,p_2,u_0\}    & \{d_0,p_0,q_2,s_2\} & \\\hline
    \{u_0,u_1,u_2,u_3\} & \{a_0,e_1,q_1,u_0\}    & \{b_0,d_3,r_2,t_3\} & \{c_0,f_3,r_1,u_3\}    & \{d_0,q_0,q_1,u_2\} & \\\hline
    \{v_0,v_1,v_2,v_3\} & \{a_0,e_2,s_3,t_2\}    & \{b_0,e_2,q_1,t_1\} & \{c_0,f_0,s_0,v_3\}    & \{d_0,s_0,u_3,v_2\} & \\\hline
                        & \{a_0,e_3,s_2,v_3\}    & \{b_0,e_3,r_0,u_1\} & \{c_0,p_0,q_0,v_1\}    & \{d_0,t_3,u_0,v_1\} & \\\hline
  \end{array}
\]}
\end{table}

\begin{table}[ht!]
\small
\caption{$4$-GDD of type $2^{18}5^{4}$} \label{21854}
\medskip
Points: $a_i,b_i,c_i,d_i,e_i,p_i,q_i,r_i,s_i,t_i,u_i,v_i,y_i,z_i$ for $i\in\Z_4$;\\
Groups: $\{a_i, b_i, c_i, d_i, e_i\}$ for $i\in\Z_4$; \, $\{p_i, q_i\}$ for $i\in\Z_4$; $\{w_i, w_{i+2}\}$ for $w \in \{r,s,t,u,v,y,z\}$ and $i \in \{0,1\}$.\\
Develop the following blocks  $(\bmod\ 4)$.
The first five blocks in the first column generate one block each and the sixth generates two blocks.
{\footnotesize
\[
  \begin{array}{|l|l|l|l|l|l|}\hline
    \{a_0,a_1,a_2,a_3\}  & \{a_0,b_1,c_2,v_0\} & \{a_0,q_2,s_1,z_1\} & \{b_0,p_1,s_2,v_0\} & \{c_0,q_3,y_1,y_2\} & \{e_0,p_1,q_3,y_0\}\\\hline
    \{b_0,b_1,b_2,b_3\} & \{a_0,b_2,p_0,r_0\} & \{a_0,q_1,u_2,v_3\} & \{b_0,q_0,u_3,z_1\} & \{c_0,r_2,s_1,u_2\} & \{e_0,p_3,v_3,y_1\}\\\hline
    \{c_0,c_1,c_2,c_3\}   & \{a_0,b_3,v_1,y_0\} & \{a_0,q_3,u_3,y_3\} & \{b_0,r_0,v_1,y_2\} & \{c_0,r_0,y_0,z_2\} & \{e_0,q_1,r_2,u_3\}\\\hline
    \{d_0,d_1,d_2,d_3\} & \{a_0,c_1,d_2,s_0\} & \{a_0,s_2,t_3,z_3\} & \{b_0,t_2,t_3,u_0\} & \{c_0,s_2,u_1,v_3\} & \{e_0,s_0,u_0,v_0\}\\\hline
    \{e_0,e_1,e_2,e_3\} & \{a_0,c_3,q_0,t_0\} & \{b_0,c_3,e_2,s_3\} & \{b_0,u_2,y_3,z_3\} & \{d_0,e_3,q_3,z_3\} & \{e_0,t_2,v_2,y_2\}\\\hline
    \{p_0,p_2,q_1,q_3\} & \{a_0,d_1,p_1,t_1\} & \{b_0,c_2,r_1,z_2\} & \{c_0,d_3,p_2,y_3\} & \{d_0,e_1,r_0,r_1\} & \{p_0,t_2,u_2,v_1\}\\\hline
                                         & \{a_0,d_3,r_1,u_0\} & \{b_0,d_2,q_2,q_3\} & \{c_0,d_2,q_0,v_0\} & \{d_0,e_2,s_1,u_3\} & \{p_0,v_2,z_0,z_1\}\\\hline
                                         & \{a_0,e_1,r_2,v_2\} & \{b_0,d_1,s_0,y_0\} & \{c_0,e_1,p_1,z_3\} & \{d_0,p_1,p_2,s_0\} & \{q_0,r_0,s_1,t_3\}\\\hline
                                         & \{a_0,e_2,t_2,y_1\} & \{b_0,d_3,t_1,z_0\} & \{c_0,e_2,t_3,z_1\} & \{d_0,r_3,t_3,v_1\} & \{q_0,r_2,s_2,v_1\}\\\hline
                                         & \{a_0,e_3,u_1,z_0\} & \{b_0,e_1,p_3,t_0\} & \{c_0,p_3,r_1,t_2\} & \{d_0,t_1,u_0,y_2\} & \{q_0,r_3,t_1,z_2\}\\\hline
                                         & \{a_0,p_2,r_3,y_2\} & \{b_0,e_3,q_1,s_1\} & \{c_0,p_0,u_0,u_3\} & \{d_0,u_2,y_1,z_2\} & \{r_0,s_2,y_1,z_0\}\\\hline
                                         & \{a_0,p_3,s_3,z_2\} & \{b_0,p_0,r_3,u_1\} & \{c_0,q_2,t_0,v_1\} & \{d_0,v_0,v_3,z_0\} & \{s_0,s_1,t_0,y_2\}\\\hline
  \end{array}
\]}
\end{table}

\begin{table}[ht]
\small
\caption{$4$-GDD of type $2^{8}5^{8}$} \label{2858}
\medskip
Points: $a_i,b_i,c_i,d_i,e_i,p_i,q_i$ for $i\in\Z_8$.\\
Groups: $\{a_i, b_i, c_i, d_i, e_i\}$ for $i\in\Z_8$; \, $\{w_i, w_{i+4}\}$ for $w \in \{p,q\}$ and $i\in\{0,1,2,3\}$.\\
Develop the following blocks  $(\bmod\ 8)$.
The  blocks in the first column generate two blocks each.
{\footnotesize
\[
  \begin{array}{|l|l|l|l|l|l|}\hline
    \{a_0,a_2,a_4,a_6\}  & \{a_0,a_1,e_2,q_0\} & \{a_0,b_1,p_0,q_1\} & \{a_0,e_6,p_6,p_7\} & \{b_0,d_5,p_4,p_6\} & \{c_0,d_3,p_5,q_3\}\\\hline
    \{b_0,b_2,b_4,b_6\} & \{a_0,a_3,c_4,d_1\}  & \{a_0,b_5,p_2,q_4\} & \{a_0,e_5,p_3,q_2\} & \{b_0,d_2,e_7,q_3\} & \{c_0,d_2,q_5,q_7\}\\\hline
    \{c_0,c_2,c_4,c_6\}   & \{a_0,b_4,b_7,d_2\} & \{a_0,c_6,d_5,e_3\} & \{b_0,b_1,c_6,q_2\} & \{b_0,e_1,p_0,q_4\} & \{c_0,e_2,e_3,q_2\}\\\hline
    \{d_0,d_2,d_4,d_6\} & \{a_0,b_2,c_3,e_4\}  & \{a_0,c_7,d_3,q_5\} & \{b_0,c_3,d_1,d_4\} & \{b_0,e_4,q_5,q_6\} & \{d_0,d_1,e_2,p_5\}\\\hline
    \{e_0,e_2,e_4,e_6\}  & \{a_0,b_3,c_5,p_5\} & \{a_0,d_4,e_7,p_4\} & \{b_0,c_7,e_5,p_1\} & \{c_0,c_1,p_4,q_1\} & \{d_0,p_3,p_6,q_6\}\\\hline
                                          & \{a_0,b_6,c_2,p_1\} & \{a_0,d_7,q_3,q_6\} & \{b_0,d_7,e_3,e_6\} & \{c_0,c_3,e_7,p_1\} & \\\hline                      
  \end{array}
\]}
\end{table}

\begin{table}[ht]
\small
\caption{$4$-GDD of type $2^{17}5^{5}$} \label{21755}
\medskip
Points: $a_i,b_i,c_i,p_i,q_i$ for $i\in\Z_{10}$; \, $y_i$ for $i\in\Z_5$; \, $\infty_1,\infty_2,\infty_3,\infty_4$.\\
Groups: $\{a_i, b_i\}$ for $i \in\Z_{10}$; \, $\{c_i, c_{i+5} \}$ for $i\in\{0,1,2,3,4\}$; \, $\{y_i: i \in \Z_5\}$;\\\phantom{Groups:}\hspace*{.33mm}
        $\{\infty_1, \infty_2\}$; \, $\{\infty_3, \infty_4\}$; \,
        $\{w_i, w_{i+2}, w_{i+4}, w_{i+6}, w_{i+8} \}$ for $w \in \{p,q\}$ and $i \in \{0,1\}$.\\
First, form a 4-GDD of type $2^7$ with groups $\{c_i, c_{i+5}\}$ for $i \in \{0,1,2,3,4\}$,
$\{\infty_1, \infty_2\}$ and $\{\infty_3, \infty_4\}$.
Next, develop the following blocks  $(\bmod\ 10)$.
The four blocks in the first column generate five blocks each.
{\footnotesize
\[
  \begin{array}{|l|l|l|l|l|l|}\hline
    \{a_0,a_5,y_0,\infty_4\}      & \{a_0,a_1,b_2,c_0\} & \{a_0,b_7,p_1,\infty_1\} & \{a_0,c_7,q_7,q_8\}      & \{b_0,b_4,c_6,p_2\} & \{b_0,c_7,q_2,y_4\}\\\hline
    \{b_0,b_5,y_0,\infty_3\} & \{a_0,a_2,b_5,y_1\} & \{a_0,b_8,q_2,\infty_2\} & \{a_0,p_3,p_4,q_4\}      & \{b_0,b_3,p_0,q_9\} & \{b_0,p_9,q_3,\infty_4\}\\\hline
    \{p_0,p_5,y_0,\infty_2\} & \{a_0,a_3,c_1,p_0\} & \{a_0,b_9,c_3,y_2\}      & \{a_0,p_9,q_5,y_3\}      & \{b_0,c_3,p_3,q_0\} & \{c_0,p_4,q_2,y_1\}\\\hline
    \{q_0,q_5,y_0,\infty_1\} & \{a_0,a_4,c_6,q_0\} & \{a_0,c_4,p_2,p_5\}      & \{a_0,p_6,q_9,\infty_3\} & \{b_0,c_9,p_1,y_2\} & \{c_0,p_7,q_9,y_0\}\\\hline
                             & \{a_0,b_4,b_6,q_1\} & \{a_0,c_5,p_8,q_3\}      & \{b_0,b_1,c_1,p_6\}      & \{b_0,c_5,q_1,q_8\} & \\\hline
  \end{array}
\]}
\end{table}

\begin{table}[ht]
\small
\caption{$4$-GDD of type $2^{11}5^{8}$} \label{21158}
\medskip
Points: $a_i,b_i,c_i,d_i,p_i,q_i$ for $i\in\Z_{10}$; \, $y_i$ for $i\in\Z_2$.\\
Groups: $\{w_i, w_{i+5}\}$ for $w \in \{p,q\}$ and $i \in \{0,1,2,3,4\}$; \, $\{y_i: i \in \Z_2\}$;\\\phantom{Groups:}\hspace*{.33mm}
        $\{w_i, w_{i+2}, w_{i+4}, w_{i+6}, w_{i+8} \}$ for $w \in \{a,b,c,d\}$ and $i \in \{0,1\}$.\\
Develop the following blocks  $(\bmod\ 10)$. The two blocks in the first column generate five blocks each.
{\footnotesize
\[
  \begin{array}{|l|l|l|l|l|l|}\hline
    \{a_0,a_5,b_0,b_5\}     & \{a_0,a_1,c_0,d_1\} & \{a_0,b_7,d_4,p_4\} & \{a_0,c_8,q_6,q_9\} & \{b_0,c_3,d_9,q_0\} & \{b_0,d_5,q_5,q_7\}\\\hline
    \{c_0,c_5,d_0,d_5\}      & \{a_0,a_3,d_2,q_0\} & \{a_0,b_8,p_2,q_1\} & \{a_0,d_5,d_8,p_9\} & \{b_0,c_6,p_0,p_9\} & \{c_0,d_8,p_1,q_2\}\\\hline
                                             & \{a_0,b_1,b_2,d_3\} & \{a_0,b_9,q_3,y_0\} & \{a_0,p_3,p_6,q_8\} & \{b_0,c_5,p_2,q_8\} & \{c_0,d_3,q_0,q_6\}\\\hline
                                             & \{a_0,b_3,c_1,c_2\} & \{a_0,c_5,d_7,p_5\} & \{a_0,p_7,q_4,q_5\} & \{b_0,c_2,q_6,y_0\} & \{c_0,p_5,p_9,q_9\}\\\hline
                                             & \{a_0,b_4,c_4,p_0\} & \{a_0,c_7,d_6,q_2\} & \{b_0,b_3,d_3,q_2\} & \{b_0,d_4,p_1,p_3\} & \{d_0,d_1,p_6,y_0\}\\\hline
                                             & \{a_0,b_6,c_3,p_1\} & \{a_0,c_6,p_8,y_1\} & \{b_0,c_1,c_4,d_8\} & \{b_0,d_6,p_8,q_1\} & \\\hline
  \end{array}
\]}
\end{table}

\begin{table}[ht]
\small
\caption{$4$-GDD of type $2^{20}5^{5}$} \label{22055}
\medskip
Points: $a_i,b_i,c_i,d_i,p_i,q_i$ for $i\in\Z_{10}$; \, $y_i$ for $i\in\Z_5$.\\
Groups: $\{w_i, w_{i+5}\}$ for $w \in \{a,b,c,d\}$ and $i \in \{0,1,2,3,4\}$; \, $\{y_i: i \in \Z_5\}$;\\\phantom{Groups:}\hspace*{.33mm}
        $\{w_i, w_{i+2}, w_{i+4}, w_{i+6}, w_{i+8} \}$ for $w \in \{p,q\}$ and $i \in \{0,1\}$.\\
Develop the following blocks  $(\bmod\ 10)$. The block in the first column generates five blocks.
{\footnotesize
\[
  \begin{array}{|l|l|l|l|l|l|}\hline
    \{p_0,p_5,q_0,q_5\} & \{a_0,a_1,b_0,c_0\} & \{a_0,b_8,d_3,q_6\} & \{a_0,d_8,p_4,y_2\} & \{b_0,c_1,p_2,q_9\} & \{c_0,c_3,d_7,y_0\}\\\hline
                                         & \{a_0,a_2,b_3,y_0\} & \{a_0,c_2,c_3,d_2\} & \{a_0,d_9,q_4,q_5\} & \{b_0,c_5,p_7,q_1\} & \{c_0,c_2,p_6,q_2\}\\\hline
                                         & \{a_0,a_3,c_1,q_0\} & \{a_0,c_4,d_5,d_7\} & \{a_0,p_8,q_9,y_4\} & \{b_0,d_1,d_8,p_5\} & \{c_0,d_5,p_3,q_5\}\\\hline
                                         & \{a_0,a_4,d_0,q_2\} & \{a_0,c_5,p_2,p_3\} & \{b_0,b_3,c_7,q_0\} & \{b_0,d_3,d_7,q_4\} & \{c_0,q_1,q_4,y_3\}\\\hline
                                         & \{a_0,b_2,b_4,p_0\} & \{a_0,c_6,p_5,q_3\} & \{b_0,b_4,c_3,y_4\} & \{b_0,d_9,p_0,q_3\} & \{d_0,d_1,q_9,y_2\}\\\hline
                                         & \{a_0,b_5,b_6,q_1\} & \{a_0,c_7,p_7,y_1\} & \{b_0,c_2,c_8,d_0\} & \{b_0,d_6,p_9,y_1\} & \\\hline
                                        & \{a_0,b_7,d_1,p_1\} & \{a_0,d_4,p_6,p_9\} & \{b_0,c_6,d_2,p_1\} & \{b_0,p_3,q_2,y_3\} & \\\hline
  \end{array}
\]}
\end{table}

\begin{table}[ht!]
\small
\caption{$4$-GDD of type $2^{14}5^{8}$} \label{21458}
\medskip
Points: $a_i,b_i,c_i,d_i,e_i,p_i,q_i,r_i,s_i$ for $i\in\Z_7$; \, $\infty_1,\infty_2,\infty_3,\infty_4,\infty_5$.\\
Groups: $\{a_i,b_i,c_i,d_i,e_i\}$, $\{p_i,q_i\}$ and $\{r_i,s_i\}$ for $i\in\Z_7$; \, $\{\infty_1,\infty_2,\infty_3,\infty_4,\infty_5\}$.\\
Develop the following blocks  $(\bmod\ 7)$.
{\footnotesize
\[
  \begin{array}{|l|l|l|l|l|l|}\hline
    \{a_0,a_1,b_2,q_0\}      & \{a_0,c_5,p_1,\infty_5\} & \{b_0,b_1,c_2,p_0\}      & \{b_0,e_3,p_3,r_6\}      & \{c_0,d_2,p_2,s_5\}      & \{e_0,e_3,p_2,p_4\}\\\hline
    \{a_0,a_2,c_1,s_0\}      & \{a_0,d_3,e_4,q_3\}      & \{b_0,b_2,d_1,s_0\}      & \{b_0,e_5,q_6,s_6\}      & \{c_0,e_2,p_0,\infty_1\} & \{e_0,e_1,s_4,s_6\}\\\hline
    \{a_0,a_3,d_1,s_2\}      & \{a_0,d_4,p_2,q_5\}      & \{b_0,b_3,p_1,r_1\}      & \{b_0,e_6,s_1,\infty_2\} & \{c_0,p_4,q_2,\infty_3\} & \{e_0,r_5,s_0,\infty_5\}\\\hline
    \{a_0,b_3,c_2,r_0\}      & \{a_0,d_6,p_3,\infty_2\} & \{b_0,c_3,d_4,q_0\}      & \{b_0,q_4,s_3,\infty_4\} & \{c_0,q_1,r_0,\infty_2\} & \{p_0,p_3,q_2,s_4\}\\\hline
    \{a_0,b_4,e_1,\infty_3\} & \{a_0,e_3,e_5,r_2\}      & \{b_0,c_4,e_1,q_3\}      & \{c_0,c_3,e_6,q_3\}      & \{d_0,d_3,e_2,r_4\}      & \{p_0,q_1,q_4,s_2\}\\\hline
    \{a_0,b_5,p_0,r_1\}      & \{a_0,e_6,q_2,q_4\}      & \{b_0,c_5,r_0,r_2\}      & \{c_0,c_2,p_1,r_3\}      & \{d_0,d_2,p_3,r_2\}      & \{q_0,q_1,r_2,r_5\}\\\hline
    \{a_0,b_6,s_1,\infty_1\}   & \{a_0,p_4,p_5,s_3\}      & \{b_0,d_5,e_2,q_2\}      & \{c_0,c_1,s_1,s_4\}      & \{d_0,e_3,p_6,r_3\}      & \{q_0,r_0,s_3,s_4\}\\\hline
    \{a_0,c_3,d_2,q_1\}      & \{a_0,p_6,r_4,\infty_4\} & \{b_0,d_2,p_4,s_4\}      & \{c_0,d_4,d_5,s_2\}      & \{d_0,q_2,r_5,\infty_1\} & \\\hline
    \{a_0,c_4,e_2,r_3\}      & \{a_0,r_5,r_6,s_4\}      & \{b_0,d_3,q_1,\infty_5\} & \{c_0,d_3,e_1,\infty_4\} & \{d_0,r_6,s_0,\infty_3\} & \\\hline
  \end{array}
\]}
\end{table}

\begin{table}[ht!]
\small
\caption{$4$-GDD of type $2^{26}5^{5}$} \label{22655}\vspace{-1mm}
\medskip
Points: $a_i,b_i,c_i,d_i,p_i,q_i,r_i$ for $i\in\Z_{10}$; $y_i$ for $i\in \Z_{2}$; \, $z_i$ for $i\in\Z_{5}$.\\
	Groups: $\{w_i, w_{i+2}, w_{i+4}, w_{i+6}, w_{i+8}\}$ for $w \in \{a,b\}$  and $i\in \{0,1\}$;\\\phantom{Groups:}\hspace*{.33mm} 
$\{w_i, w_{i+5}\}$ for $w \in \{c,d,p,q,r\}$ and $i\in \{0,1,2,3,4\}$; \,   $\{y_i: i \in  \Z_{2} \}$; \, $\{z_i: i \in \Z_5\}$. \\
Develop the following blocks  $(\bmod\ 10)$.  The block in the first column generates five blocks.
{\footnotesize
\[ 
  \begin{array}{|l|l|l|l|l|l|}\hline
\{a_0,a_5,b_0,b_5\}      & \{a_0,a_3,c_6,q_2\}      & \{a_0,c_2,d_5,q_1\}      & \{a_0,q_6,y_1,z_4\}      & \{b_0,p_4,q_4,z_4\}      & \{c_0,p_0,q_7,r_5\}\\\hline
                         & \{a_0,a_9,d_1,z_2\}      & \{a_0,c_9,p_0,r_3\}      & \{b_0,b_1,c_0,q_0\}      & \{b_0,p_7,r_8,y_1\}      & \{c_0,q_2,q_4,r_7\}\\\hline
                         & \{a_0,b_4,c_5,d_6\}      & \{a_0,d_4,p_1,y_0\}      & \{b_0,b_3,d_6,p_5\}      & \{b_0,q_3,r_7,z_2\}      & \{c_0,q_1,r_0,z_3\}\\\hline
                         & \{a_0,b_3,c_7,p_3\}      & \{a_0,d_7,q_4,r_1\}      & \{b_0,c_8,d_7,q_6\}      & \{b_0,r_4,r_6,z_0\}      & \{d_0,d_4,p_0,q_2\}\\\hline
                         & \{a_0,b_6,c_1,p_5\}      & \{a_0,d_8,q_3,r_5\}      & \{b_0,c_6,p_8,r_5\}      & \{c_0,c_4,d_2,d_4\}      & \{d_0,d_1,p_4,z_0\}\\\hline
                         & \{a_0,b_1,c_8,r_0\}      & \{a_0,d_9,r_4,r_7\}      & \{b_0,d_8,p_3,z_1\}      & \{c_0,c_1,d_7,p_9\}      & \{d_0,d_3,r_3,r_9\}\\\hline
                         & \{a_0,b_2,c_4,r_2\}      & \{a_0,p_2,p_6,q_5\}      & \{b_0,d_4,q_5,q_8\}      & \{c_0,c_2,q_5,z_1\}      & \{p_0,p_3,q_4,q_8\}\\\hline
                         & \{a_0,b_7,c_0,z_0\}      & \{a_0,p_7,p_8,r_6\}      & \{b_0,d_9,q_2,y_0\}      & \{c_0,c_3,r_6,y_0\}      & \\\hline
                         & \{a_0,b_8,d_3,p_4\}      & \{a_0,p_9,r_9,z_1\}      & \{b_0,d_0,r_1,r_2\}      & \{c_0,d_5,p_3,z_2\}      & \\\hline
                         & \{a_0,b_9,d_0,q_0\}      & \{a_0,q_7,q_8,r_8\}      & \{b_0,p_1,q_7,r_3\}      & \{c_0,p_5,p_7,r_1\}      & \\\hline
  \end{array}
\]}
\end{table}

\begin{table}[ht!]
\small
\caption{$4$-GDD of type $2^{16}5^{9}$} \label{21659}\vspace{-1mm}
\medskip
Points: $a_i,b_i,c_i,d_i,p_i,q_i,r_i$ for $i\in\Z_{10}$; $y_i$ for $i\in \Z_{2}$; \, $z_i$ for $i\in\Z_{5}$.\\
	Groups: $\{w_i, w_{i+2}, w_{i+4}, w_{i+6}, w_{i+8}\}$ for $w \in \{a,b,c,d\}$  and $i\in \{0,1\}$;\\\phantom{Groups:}\hspace*{.33mm} 
$\{w_i, w_{i+5}\}$ for $w \in \{p,q,r\}$ and $i\in \{0,1,2,3,4\}$; \,   $\{y_i: i \in  \Z_{2} \}$; \, $\{z_i: i \in \Z_5\}$. \\
Develop the following blocks  $(\bmod\ 10)$.  The two blocks in the first column generate five blocks each.
{\footnotesize
\[ 
  \begin{array}{|l|l|l|l|l|l|}\hline
\{a_0,a_5,b_0,b_5\}      & \{a_0,a_7,c_4,q_0\}      & \{a_0,c_8,d_7,q_5\}      & \{a_0,q_9,y_1,z_4\}      & \{b_0,p_4,q_1,z_3\}      & \{c_0,p_6,q_0,r_2\}\\\hline
\{c_0,c_5,d_0,d_5\}      & \{a_0,a_1,d_4,z_3\}      & \{a_0,c_9,p_3,y_0\}      & \{b_0,b_1,c_5,q_0\}      & \{b_0,p_3,r_5,y_1\}      & \{c_0,q_1,q_3,r_9\}\\\hline
                         & \{a_0,b_2,c_3,d_0\}      & \{a_0,d_2,p_4,r_1\}      & \{b_0,b_3,d_0,p_5\}      & \{b_0,q_5,r_9,z_1\}      & \{c_0,q_4,r_7,z_3\}\\\hline
                         & \{a_0,b_9,c_1,p_0\}      & \{a_0,d_5,q_2,r_3\}      & \{b_0,c_6,d_4,q_8\}      & \{b_0,r_2,r_6,z_0\}      & \{d_0,d_1,p_8,q_1\}\\\hline
                         & \{a_0,b_1,c_0,p_1\}      & \{a_0,d_8,q_4,r_9\}      & \{b_0,c_0,p_7,r_1\}      & \{c_0,c_3,d_4,p_8\}      & \{d_0,d_3,p_3,z_3\}\\\hline
                         & \{a_0,b_3,c_6,r_0\}      & \{a_0,d_9,r_4,r_5\}      & \{b_0,d_2,p_8,z_4\}      & \{c_0,c_1,q_9,z_2\}      & \{p_0,p_4,q_6,q_9\}\\\hline
                         & \{a_0,b_4,c_2,r_2\}      & \{a_0,p_5,p_6,q_6\}      & \{b_0,d_3,q_2,q_6\}      & \{c_0,d_3,p_2,z_4\}      & \\\hline
                         & \{a_0,b_8,c_5,z_0\}      & \{a_0,p_7,p_9,r_8\}      & \{b_0,d_1,q_3,y_0\}      & \{c_0,d_6,r_6,r_8\}      & \\\hline
                         & \{a_0,b_6,d_1,p_2\}      & \{a_0,p_8,r_6,z_1\}      & \{b_0,d_6,r_0,r_3\}      & \{c_0,d_2,r_5,y_0\}      & \\\hline
                         & \{a_0,b_7,d_6,q_1\}      & \{a_0,q_7,q_8,r_7\}      & \{b_0,p_9,q_7,r_4\}      & \{c_0,p_0,p_3,r_3\}      & \\\hline
  \end{array}
\]}
\end{table}

\begin{table}[ht!]
\small
\caption{$4$-GDD of type $2^{6}5^{13}$} \label{26513}\vspace{-1mm}
\medskip
Points: $a_i,b_i,c_i,d_i,p_i,q_i,r_i$ for $i\in\Z_{10}$; \, $y_i$ for $i\in \Z_{2}$; \, $z_i$ for $i\in\Z_{5}$.\\
	Groups: $\{w_i, w_{i+2}, w_{i+4}, w_{i+6}, w_{i+8}\}$ for $w \in \{a,b,c,d,p,q\}$  and $i\in \{0,1\}$;\\\phantom{Groups:}\hspace*{.33mm} 
$\{r_i, r_{i+5}\}$ for  $i\in \{0,1,2,3,4\}$; \,   $\{y_i: i \in  \Z_{2} \}$; \, $\{z_i: i \in \Z_5\}$. \\
Develop the following blocks  $(\bmod\ 10)$.  The three blocks in the first column generate five blocks each.
{\footnotesize
\[ 
  \begin{array}{|l|l|l|l|l|l|}\hline
\{a_0,a_5,b_0,b_5\}      & \{a_0,a_3,c_6,q_0\}      & \{a_0,c_2,d_5,q_2\}      & \{a_0,q_8,y_1,z_4\}      & \{b_0,p_1,q_7,z_0\}      & \{c_0,p_3,q_6,r_5\}\\\hline
\{c_0,c_5,d_0,d_5\}      & \{a_0,a_9,d_3,z_2\}      & \{a_0,c_9,p_0,y_0\}      & \{b_0,b_1,c_2,q_5\}      & \{b_0,p_6,r_0,y_1\}      & \{c_0,p_7,q_9,r_0\}\\\hline
\{p_0,p_5,q_0,q_5\}      & \{a_0,b_8,c_5,d_7\}      & \{a_0,d_2,p_5,r_1\}      & \{b_0,b_3,d_8,p_5\}      & \{b_0,q_8,r_4,z_3\}      & \{c_0,q_2,r_9,z_1\}\\\hline
                         & \{a_0,b_4,c_4,p_2\}      & \{a_0,d_6,q_4,r_2\}      & \{b_0,c_8,d_4,q_3\}      & \{b_0,r_1,r_7,z_2\}      & \{d_0,d_3,p_5,q_4\}\\\hline
                         & \{a_0,b_3,c_7,p_6\}      & \{a_0,d_8,q_3,r_5\}      & \{b_0,c_3,p_9,r_9\}      & \{c_0,c_1,d_9,p_5\}      & \{d_0,d_1,p_0,z_2\}\\\hline
                         & \{a_0,b_1,c_0,r_3\}      & \{a_0,d_9,r_4,r_7\}      & \{b_0,d_6,p_0,z_1\}      & \{c_0,c_3,q_1,z_3\}      & \\\hline
                         & \{a_0,b_2,c_8,r_0\}      & \{a_0,p_4,p_7,q_5\}      & \{b_0,d_7,q_0,q_9\}      & \{c_0,d_4,p_2,z_2\}      & \\\hline
                         & \{a_0,b_6,c_1,z_0\}      & \{a_0,p_8,p_9,r_6\}      & \{b_0,d_0,q_6,y_0\}      & \{c_0,d_7,r_7,r_8\}      & \\\hline
                         & \{a_0,b_7,d_0,p_1\}      & \{a_0,p_3,r_8,z_1\}      & \{b_0,d_1,r_3,r_5\}      & \{c_0,d_1,r_4,y_0\}      & \\\hline
                         & \{a_0,b_9,d_1,q_1\}      & \{a_0,q_6,q_9,r_9\}      & \{b_0,p_7,q_1,r_6\}      & \{c_0,p_0,q_7,r_1\}      & \\\hline
  \end{array}
\]}
\end{table}

\begin{table}[ht!]
\small
\caption{$4$-GDD of type $2^{19}5^{9}$} \label{21959}
\medskip
Points: $a_i,b_i,c_i,d_i,e_i,p_i,q_i,r_i,s_i$ for $i\in\Z_9$; \, $\infty_1,\infty_2$.\\
Groups: $\{a_i,b_i,c_i,d_i,e_i\}$, $\{p_i,q_i\}$ and $\{r_i,s_i\}$ for $i\in\Z_9$; \, $\{\infty_1,\infty_2\}$.\\
Develop the following blocks  $(\bmod\ 9)$.
{\footnotesize
\[
  \begin{array}{|l|l|l|l|l|l|}\hline
    \{a_0,a_1,b_2,s_0\} & \{a_0,c_8,s_5,\infty_2\} & \{b_0,b_1,c_2,r_5\} & \{b_0,d_6,p_7,\infty_2\} & \{c_0,e_3,e_6,r_2\} & \{e_0,e_4,s_5,s_8\}\\\hline
    \{a_0,a_2,b_5,s_4\} & \{a_0,d_4,d_5,p_2\} & \{b_0,b_2,d_1,r_1\} & \{b_0,e_4,p_5,s_1\} & \{c_0,e_4,p_0,r_4\} & \{e_0,p_3,p_4,q_8\}\\\hline
    \{a_0,a_3,c_1,p_0\} & \{a_0,d_7,e_3,r_6\} & \{b_0,b_3,e_1,p_0\} & \{b_0,e_5,q_8,\infty_1\} & \{c_0,p_3,q_1,q_6\} & \{e_0,q_2,r_2,\infty_2\}\\\hline
    \{a_0,a_4,d_1,q_0\} & \{a_0,d_8,p_3,\infty_1\} & \{b_0,b_4,q_0,s_0\} & \{b_0,p_2,p_8,s_3\} & \{c_0,q_3,s_7,s_8\} & \{p_0,p_2,r_1,s_8\}\\\hline
    \{a_0,b_4,c_2,q_1\} & \{a_0,e_4,e_5,q_2\} & \{b_0,c_3,c_5,p_1\} & \{b_0,q_4,r_0,r_6\} & \{c_0,r_5,s_2,\infty_1\} & \{p_0,p_4,r_7,s_2\}\\\hline
    \{a_0,b_6,d_2,r_0\} & \{a_0,e_6,e_8,s_6\} & \{b_0,c_4,e_2,q_2\} & \{c_0,c_4,d_3,p_6\} & \{d_0,d_4,e_6,q_1\} & \{q_0,r_3,s_6,s_8\}\\\hline
    \{a_0,b_7,e_1,p_1\} & \{a_0,e_7,p_4,q_3\} & \{b_0,c_8,q_1,q_3\} & \{c_0,c_3,e_2,p_4\} & \{d_0,e_7,p_5,r_2\} & \\\hline
    \{a_0,b_8,r_1,s_3\} & \{a_0,p_5,q_6,q_7\} & \{b_0,c_6,r_7,s_6\} & \{c_0,d_1,d_4,s_4\} & \{d_0,e_8,r_5,r_6\} & \\\hline
    \{a_0,c_3,c_4,r_2\} & \{a_0,p_7,q_4,s_7\} & \{b_0,d_2,d_4,p_4\} & \{c_0,d_6,e_1,s_3\} & \{d_0,p_8,r_1,s_2\} & \\\hline
    \{a_0,c_5,d_3,s_1\} & \{a_0,p_8,r_4,r_8\} & \{b_0,d_3,e_6,q_7\} & \{c_0,d_5,q_5,r_0\} & \{d_0,q_2,q_5,r_3\} & \\\hline
    \{a_0,c_6,e_2,r_3\} & \{a_0,q_8,r_5,r_7\} & \{b_0,d_7,e_8,s_2\} & \{c_0,d_2,q_0,s_1\} & \{d_0,q_3,s_1,s_5\} & \\\hline
  \end{array}
\]}
\end{table}

\begin{table}[ht]\small
  \caption{ $4$-GDD of type $5^8 14^1 20^1$}   \label{58141201}
\medskip
  Points: $a_i,  b_i, c_i$  for  $i \in \Z_{20}$; $\infty_1, \infty_2, \ldots, \infty_{14}$.\\
  Groups: $\{w_i, w_{i+4}, w_{i+8}, w_{i+12}, w_{i+16}\}$ for $w \in \{a,b\}$ and $i \in \{0,1,2,3\}$;  $\{c_i:  i \in \Z_{20}\}$; \\\phantom{Groups:}\hspace*{.33mm}
  $\{\infty_1, \infty_2, \ldots, \infty_{14}\}$.
 \newline Develop the following blocks  $(\bmod\ 20)$. The  block in the first column generates ten blocks.
{\noindent
{\footnotesize
\[
  \begin{array}{|l|l|l|l|l|l|}\hline
    \{a_0, a_{10}, b_0, b_{10}\}   &  \{a_0, a_7, b_9, b_{18}\}        &  \{b_4, b_{17},b_{18},c_0\}      &  \{a_{7}, b_{12}, c_0,\infty_4\} &  \{a_{15}, b_{7}, c_0,\infty_8\}  & \{a_{17},b_{13},c_0,\infty_{12}\} \\\hline
                                   & \{a_0, a_5, a_{11}, c_0\}      & \{a_{18},b_{19},c_0,\infty_1\}  &  \{a_{9}, b_{15},c_0,\infty_5\}  &  \{a_{12}, b_{5},c_0,\infty_9\}   & \{a_{14},b_{11},c_0,\infty_{13}\} \\\hline
                                   & \{a_3, a_4, a_6, c_0\}     & \{a_{13},b_{16},c_0,\infty_2\}  &  \{a_{1}, b_{ 8},c_0,\infty_6\}  & \{a_{8}, b_{ 2},c_0,\infty_{10}\} & \{a_{10},b_{ 9},c_0,\infty_{14}\} \\\hline
                                   & \{b_1, b_3, b_6, c_0\}         &  \{a_{16},b_{0},c_0,\infty_3\}  &   \{a_{2},b_{10},c_0,\infty_7\}  & \{a_{19},b_{14},c_0,\infty_{11}\} &   \\\hline
 \end{array}
\]}}
\end{table} 

\vfill

\begin{thebibliography}{99}
%\frenchspacing

%\bibitem{mols1860}
%R.J.R. Abel,  Five MOLS of orders 18 and 60,  J. Combin. Des.  27 (2015), 135--139.

%\bibitem{ABCD}
%R.J.R.  Abel,   C.J. Colbourn and J.H. Dinitz, Mutually orthogonal Latin 
%squares (MOLS), in:  The CRC Handbook of Combinatorial Designs, Second Edition
%(C.  J.\ Colbourn and J.  H.\ Dinitz, eds.), CRC Press, Boca Raton FL, 2007, 160--193.

\bibitem{5423}
R.J.R. Abel,  Y.A. Bunjamin and D. Combe, Some new group divisible designs with block size $4$ and two or three group sizes,  J. Combin. Des. 28 (2020), 614--628.  

\bibitem{ABC.50less}
R.J.R. Abel, Y.A. Bunjamin and  D. Combe, Existence of 4-GDDs with at most 50 points and 4-GDDs of types $6^s 3^t$ and $9^s 3^t$, Discrete Math. 344 (2021), 112479.

\bibitem{ABBC.2t8s}
R.J.R. Abel, T. Britz, Y.A. Bunjamin and  D. Combe, Group divisible designs with block size 4 where the group sizes are congruent to 2 mod 3, Discrete Math. 345 (2022), 112740.

%\bibitem{BR1979} A.E. Brouwer, Optimal packings of $K_4$'s into a %$K_n$. J. Combin. Theory Ser. A 26 (1979), no. 3, 278--297.


%\bibitem{BW}
%R.D. Baker  and R.M. Wilson, Nearly Kirkman triple systems,
%    Util. Math.  11 (1977), 289--296.

\bibitem{BSH}
A.E. Brouwer, A. Schrijver and H. Hanani,  Group divisible designs with block size 4, Discrete Math. 20 (1977), 1--10.


%\bibitem{dengreesshen3} D. Deng, R. Rees and H. Shen, Further results on nearly Kirkman triple systems with subsystems,  Discrete Math. 270, (2003),  99--114.

\bibitem{Forbes1}
A.D. Forbes and K.A. Forbes,  Group divisible designs with block size 4 and type $g^u m^1$,  J. Combin. Des.  26 (2018), 519--539.

\bibitem{Forbes2}
A.D. Forbes,  Group divisible designs with block size 4 and type $g^u m^1$ II,  J. Combin. Des.  27 (2019), 311--349.

\bibitem{Forbes3}
A.D. Forbes,  Group divisible designs with block size 4 and type $g^u m^1$ III,  J. Combin. Des.  27 (2019), 643--672.

%\bibitem{Forbesgcc}
%A.D. Forbes,  Group divisible designs with block size 4 and type $g^u b^1 (gu/2)^1$,  Graphs Combin. 36 (2020), 1687--1703. % https://doi.org/10.1007/s00373-020-02213-5.

\bibitem{fmy}   
S.C. Furino, Y. Miao and J. Yin, Frames and resolvable designs, CRC Press, Boca Raton FL, 1996.

\bibitem{gerees}
G. Ge and R.S. Rees, On group-divisible designs with block size four and group type $g^u m^1$, Des. Codes Cryptogr. 27 (2002), 5--24.

%\bibitem{ge0mod6}
% G. Ge and H. Wei,  Group divisible designs with block size 4  and group type 
% $g^u m^1$ for $g \equiv 0$ $($mod $6)$  J. Combin Des. 22  (2014), 26--52.  

\bibitem{gezhu}
G. Ge, R.S. Rees and L. Zhu, Group divisible designs with block size 4 and group type $g^u m^1$ with $m$ as small as possible,  J. Combin. Theory Ser. A  98 (2002), 357-376.  

%\bibitem{reesgrut}
%M. Gr\"{u}ttm\"{u}ller and R.S. Rees, Mandatory representation designs MRD$(4,k;v)$ with $k \equiv 1 \pmod{3}$, Util.  Math. 60 (2001), 153--180.   %R.S. Rees here

%   a\bibitem{Ingo}
% I. Janiszczak and R. Staszewski,  Isometry invariant permutation codes and mutually orthogonal
%  Latin squares,    J. Combin. Des.  27 (2019), 541--551.

\bibitem{krestin}
D. Kreher and D.R. Stinson, Small group divisible designs with block size four, J. Statist. Plann. Inference 58 (1997), 111--118.

\bibitem{mckay.nauty}
B.D. McKay and A. Piperno, Practical Graph Isomorphism, II,  J. Symbolic Comput. 60 (2014), 94--112.

\bibitem{reesk=45}
R.S. Rees, Group-divisible designs with block size $k$ having $k+1$ groups for $k=4,5$, J. Combin. Des. 8 (2000), 363--386. %R.S. Rees here

%\bibitem{reesstinsubs}
%R. Rees  and D.R. Stinson,  Kirkman triple systems with maximum subsystems, Ars Combin.  25 (1988), 125--132.

%\bibitem{reesstin89}   % at present unciteds; it contains 4-GDD$(3^{11} 9^1)$
%R. Rees  and D.R. Stinson,  On combinatorial designs with subdesigns, Discrete Math.  77 (1989), 259--279.

% \bibitem{reesstinonehole}
% R. Rees  and D.R. Stinson,  On the existence of incomplete designs of block size 4 having one hole, Util. Math. 35 (1989), 119--152.

%\bibitem{reesstinrgdd3}
%R. Rees and D.R. Stinson, On resolvable group-divisible designs with block size $3$, Ars Combin. 23 (1987), 107--120. % R. Rees not R.S. Rees in this title

%\bibitem{ICKPD4}
%L. Wang, R.J.R. Abel, D. Deng and J. Wang, Existence of incomplete canonical Kirkman packing designs, Discrete Math. 341 (2018), 536--554.

%\bibitem{wangshen}
%J. Wang and  H. Shen,  Existence of $(v, K_{1(3)} \cup \{w^*\})$-PBDs and its applications, Des. Codes Cryptogr. 46 (2008), 1--16.  

%\bibitem{geweikf}
%H. Wei and G. Ge,  Group divisible designs with block sizes from $K_{1(3)}$ and Kirkman frames of type $g^u m^1$,  Discrete Math.  329 (2014), 42--68.  

\bibitem{gegum}
H. Wei and G. Ge, Group divisible designs with block size 4 and type $g^u m^1$, Des. Codes Cryptogr. 74 (2015), 243--282.  

%\bibitem{Wilson}
%R.M. Wilson, Constructions and uses of pairwise balanced designs, Math. Centre Tracts  55 (1974), 18--41.

\end{thebibliography}
\end{document}